\documentclass[12pt]{elsarticle}
\usepackage{amscd,amssymb,amsmath,amsthm}
\usepackage[all]{xy}
\usepackage{array}
\usepackage{etoolbox}
\usepackage{mathtools}
\usepackage{romannum}
\usepackage[most]{tcolorbox}
\usepackage{tikz-cd}

\makeatletter
\patchcmd{\ps@pprintTitle}{\footnotesize\itshape
        \hfill\today}{\relax}{}{}

\newdir{ >}{!/8pt/\dir{}*\dir{>}}
\newtheorem{theorem}{Theorem}[section]
\newtheorem*{theorem*}{Theorem}

\newtheorem*{outline}{Outline of the paper}
\newtheorem{lemma}[theorem]{Lemma}
\newtheorem{corollary}[theorem]{Corollary}

\newtheorem*{corollary*}{Corollary}
\newtheorem*{main corollary*}{Main Corollary}
\newtheorem{prop}[theorem]{Proposition}

\newtheorem*{con7*}{Conjecture 7*}
\newtheorem{Remark}{Remark}
\theoremstyle{definition}
\newtheorem{definition}[theorem]{Definition}

\newtheorem*{Notations}{Notations}

\newcommand\dela[1]{}
\newcommand{\tens}{
	\otimes^q}
\newcommand{\ptens}{
	\otimes^p}
 \newcommand{\im}[1]{\text{Im}{(#1)}}

\usepackage[tickmarkheight=0.1cm]{todonotes}

   \newcommand{\settheoremtag}[1]{
  \let\oldthetheorem\thetheorem
  \renewcommand{\thetheorem}{#1}
  \g@addto@macro\endtheorem{
    \global\let\thetheorem\oldthetheorem}
  }
\AtBeginDocument{\pagenumbering{arabic}}

\makeatletter
\def\ps@pprintTitle{%
	\let\@oddhead\@empty
	\let\@evenhead\@empty
	\def\@oddfoot{\reset@font\hfil\thepage\hfil}
	\let\@evenfoot\@oddfoot
	
}

\makeatother

\journal{}

\begin{document}

\begin{frontmatter}

\title{Powerfully embedded subgroups of extensions of powerful pro-$p$ groups and its applications to automorphisms of finite groups}

 \author[IISER TVM]{Sathasivam Kalithasan}
\ead{sathasivam19@iisertvm.ac.in}
 \author[IISER TVM]{Tony N. Mavely}
\ead{tonynixonmavely17@iisertvm.ac.in}
\author[IISER TVM]{Viji Z. Thomas\corref{cor1}}
\address[IISER TVM]{School of Mathematics,  Indian Institute of Science Education and Research Thiruvananthapuram,\\695551
Kerala, India.}
\ead{vthomas@iisertvm.ac.in}
\cortext[cor1]{Corresponding author. \emph{Phone number}: +91 8921458330}

\begin{abstract}

One of the aims of this paper is to obtain structural results showing that powerful subgroups are abundant in pro-$p$ groups admitting certain powerful quotients. In particular, we obtain an analogue of Baer’s theorem for powerful pro-$p$ groups, namely that the powerfulness of $H/Z_{n-1}(H)$ implies that the $n$th terms of both the lower $p$-series and the lower central series of $H$ are powerfully embedded in $H$. As a consequence, we obtain that if $H$ is a finitely generated pro-$p$ group and $H/Z_n(H)$ is a $p$-adic analytic pro-$p$ group for some positive integer $n$, then $H$ is a $p$-adic analytic pro-$p$ group. We also study crossed squares of powerful $p$-groups, establishing that if $\mu : M \to G$ is a crossed module with $M$ a finite powerful $p$-group and $G$ a finite $p$-group, and if $\mu(M)$ is powerfully embedded in $G$, then both $M \otimes G$ and $M \otimes^{p} G$ are powerful. Generalizing the commutator and the Frattini subgroups via the action of automorphisms, we study relative autocommutator and the generalized Frattini subgroups $L_A(G)$ and $\Phi_A(G)$, respectively, for subgroups $A$ with $\operatorname{Inn}(G) \leqslant A \leqslant \operatorname{Aut}(G)$. When $G$ is a finite powerful $p$-group and $A$ is chosen from $\operatorname{Aut}_\Phi(G)$, $\operatorname{Aut}_c(G)$, or $\operatorname{IA}(G)$, we show that $\operatorname{Inn}(G)$ is powerfully embedded in $A$. As a consequence, both $L_A(G)$ and $\Phi_A(G)$ are powerful subgroups of $G$.

\end{abstract}

\begin{keyword}
Pro-$p$ groups \sep $p$-adic analytic groups \sep powerful $p$-groups \sep Lazard series \sep Frattini series \sep Crossed modules \sep Automorphism groups \sep Autocommutator subgroup. 

\MSC[2020] 18G45 \sep 20D45 \sep	20E18 \sep 20F28.
\end{keyword}

\end{frontmatter}

 \section{Introduction}
 One of the main aims of this paper is to prove

\settheoremtag{A}
\begin{theorem} \label{MR: H/Z,H/D powerful gamma_n,P_n,Gamma_n powerful}
Let $p$ be an odd prime and $H$ be a pro-$p$ group. Then the following statements hold:
\begin{enumerate}
\item[(i)]  If $H/Z_{n-1}(H)$ is powerful, then the closure of the $n$th term of the lower central series of $H$ is powerfully embedded in $H$.
\item[(ii)]  If $H/Z_{n-1}(H)$ is powerful, then the $n$th term of the lower $p$-series (Lazard series) of $H$ is powerfully embedded in $H$.
\item [(iii)] If $H/\mathcal{D}_{n-1}(H)$ (cf. Definition \ref{D_n(G)}) is powerful, then the closure of the $n$th term of the derived series of $H$ is powerfully embedded in $H$.
\end{enumerate}
\end{theorem}

\settheoremtag{A$'$}
\begin{theorem} \label{MR: foreven H/Z,H/D powerful gamma_n,P_n,Gamma_n powerful}
Let $p=2$ and $H$ be a pro-$p$ group. Then the following statements hold:
\begin{enumerate}
\item[(i)]  If $H/Z_{n-1}(H)$ is powerful, then the closure of the $n$th term of the lower central series of $H$ is powerfully embedded in $H$. 
\item[(ii)]  If $H/Z_{n-1}(H)$ is powerful, then the $n$th term of the lower $p$-series (Lazard series) of $H$ is powerful.
\end{enumerate}
\end{theorem}

As an easy consequence of the above Theorems, we obtain the following elegant result.

\begin{main corollary*} \label{MR: p-adic analytic corollary}
Let $p$ be a prime, and $H$ be a finitely generated pro-$p$ group. If $H/Z_n(H)$ is $p$-adic analytic for some $n\in \mathbb{N}$, then so is $H$.
\end{main corollary*}


In \cite[Theorem 3.4.5]{La1965}, Serre proved that the class of Hausdorff topological groups that are $p$-adic analytic is closed under extensions: if $H$ is a Hausdorff topological group with a closed normal subgroup $N$ such that both $N$ and $H/N$ are $p$-adic analytic, then $H$ is $p$-adic analytic. In the above corollary, we do not require that the kernel is $p$-adic analytic. This corollary is particularly useful in light of the characterizations of ${p\text{-adic analytic groups}}$. Specifically, to verify any of the equivalent conditions listed below for a group $H$, it suffices to verify the condition for the quotient $H/Z_n(H)$ for some $n \in \mathbb{N}$.

\settheoremtag{1.1}
\begin{theorem}\cite[Interlude A, 1]{DDMS99} \label{padic analytic characterization}
    Let $G$ be a pro-$p$ group. The following conditions are equivalent:
\begin{enumerate}
    \item [(i)] $G$ is a $p$-adic analytic group.
    \item [(ii)] $G$ has finite rank.
    \item [(iii)] $G$ is finitely generated and virtually powerful.
    \item [(iv)] $G$ is finitely generated and virtually uniform.
    \item [(v)] $G$ has polynomial subgroup growth.
    \item [(vi)] $G$ is the product of finitely many procyclic subgroups.
    \item [(vii)] $G$ is isomorphic to a closed subgroup of $GL_d(\mathbb{Z}_p)$ for some $d$.
\end{enumerate}
\end{theorem}

In the next theorem, we relax the hypothesis that $H$ is a pro-$p$ group.

\settheoremtag{B}
\begin{theorem} \label{MR: extension of powerful implies lower central and derived are powerful}
    Let $p$ be an odd prime, suppose we have a group extension $1\to N\to H \to G\to 1$, where $G$ is a finite powerful $p$-group. The following statements hold:
    
    \begin{enumerate}
        \item [(i)] If $N$ is contained in the $n$th center of $H$ for some positive integer $n$, then the $(n+1)$-th term of the lower central series of $H$ is finite powerful $p$-group.
       \item[(ii)] If $N\leqslant \mathcal{D}_n(H)$ (cf. Definition \ref{D_n(G)}) for some positive integer $n$, then the $(n+1)$-th term of derived series of $H$ is finite powerful $p$-group. 
    \end{enumerate}
\end{theorem}

In this paper, we use a novel approach of studying powerful $p$-groups using crossed module squares and associated cubes.  A crossed module is a group homomorphism $\mu: M \to G$ together with an action of $G$ on $M$, which satisfies the following conditions
\begin{align}
    &\mu(^g m)=g\mu(m)g^{-1} , g\in G, m\in M\nonumber\\
        &^{\mu(m)}m'=mm'm^{-1}, m,m'\in M \nonumber.
\end{align}
Powerful $p$-groups naturally arise in the study of analytic pro-$p$ groups. Since the foundational work of Lubotzky and Mann \cite{LM87, LM872}, these groups have played a significant role in the theory of finite $p$-groups and pro-$p$ groups. A pro-$p$ group $H$ is powerful, if either $p$ is odd and $[H,H] \leqslant \overline{H^p}$ or $p=2$ and $[H,H] \leqslant \overline{H^4}$. A closed normal subgroup $N$ of a pro-$p$ group $H$ is powerfully embedded in $H$ if either $p$ is odd and $[N, H] \leqslant \overline{N^p}$ or $p=2$ and $[N, H] \leqslant \overline{N^4}$.

Having defined this, we state our next result
\settheoremtag{C}
\begin{theorem} \label{MR: M tensor G and M ptensor G are powerful}
    Let $p$ be an odd prime, $M$ and $G$ be finite $p$-groups, and ${\mu:M\to G}$ be a crossed module. If $M$ is powerful and $\mu(M)$ is powerfully embedded in $G$, then $M\otimes G$ and $M\ptens G$ are powerful.
\end{theorem}
As a consequence, we obtain the following result for $n$-fold tensor product $G^{\otimes^p_{n+1}}$ (cf. \eqref{n fold qtensorsquare1}, \eqref{q n fold tensorsquare}  ) and the iterated tensor product $G_{\otimes^p_{n+1}}$ {(cf. \eqref{iteratedqtensordefn})}. 

\begin{corollary*}  \label{MR: Tensors are powerful}
	Let $p$ be an odd prime. Suppose $G$ is a finite powerful $p$-group, then $G^{\otimes^p_{n+1}}$ and $G_{\otimes^p_{n+1}}$ are finite powerful $p$-groups for all $n\geq 1$. 
\end{corollary*} 

Recall that the commutator subgroup
$\gamma_2(G)$ and the Frattini subgroup $\Phi(G)$ can be described in
terms of the action of the inner automorphism group $\operatorname{Inn}(G)$
on $G$. More precisely,
\[
\gamma_2(G)
=
\left\langle
\alpha(g)g^{-1}
\;\middle|\;
g\in G,\ \alpha\in\operatorname{Inn}(G)
\right\rangle,
\]
and
\[
\Phi(G)
=
\left\langle
g^p,\ \alpha(g)g^{-1}
\;\middle|\;
g\in G,\ \alpha\in\operatorname{Inn}(G)
\right\rangle.
\]

These descriptions suggest a natural generalization obtained by replacing
$\operatorname{Inn}(G)$ with a subgroup of $\operatorname{Aut}(G)$
containing $\operatorname{Inn}(G)$. In particular, by taking $\operatorname{Aut}(G)$, one obtains the
autocommutator subgroup of $G$, defined by $L(G)=
\left\langle
\alpha(g)g^{-1}
\;\middle|\;
g\in G,\ \alpha\in\operatorname{Aut}(G)
\right\rangle.$ More generally, let $\operatorname{Inn}(G)\leq A\unlhd\operatorname{Aut}(G).$
We define the autocommutator subgroup of $G$ corresponding to $A$ by $L_A(G)
=
\left\langle
\alpha(g)g^{-1}
\;\middle|\;
g\in G,\ \alpha\in A
\right\rangle.$ In analogy with the Frattini subgroup, we define the Frattini subgroup corresponding to $A$ by $\Phi_A(G)=
\left\langle
\alpha(g)g^{-1},\,g^p
\;\middle|\;
g\in G,\ \alpha\in A
\right\rangle.$ Thus $\Phi_A(G)=L_A(G)G^p$.
We investigate three subgroups of $\operatorname{Aut}(G)$. By applying Theorem \ref{MR: M tensor G and M ptensor G are powerful}, we prove that for each such subgroup $A$, the groups $L_A$ and $\Phi_A$ are powerful provided that $G$ is a finite powerful $p$-group. In particular, we prove that

\settheoremtag{D}
\begin{theorem} \label{MR: Inn G pe Automorphism subgroups and L_A and Phi_A powerful}
Let $G$ be a finite powerful $p$-group. 
Let $A$ denote any of the subgroups $\operatorname{Aut}_\Phi(G),\; \operatorname{Aut}_c(G),\; \operatorname{IA}(G)$ (cf. Section \ref{section: General version of commutator and Frattini are powerful}) of $\operatorname{Aut}(G)$. Then the following statements hold
\begin{enumerate}
\item[(i)]  $\operatorname{Inn}(G)$ is powerfully embedded in $A$.
\item[(ii)] $L_A$ and $\Phi_A$ are powerful subgroups of $G$.
\end{enumerate}
\end{theorem}

In \cite{Baer52}, Baer proved that if $H/Z_n(H)$ is finite, then $\gamma_{n+1}(H)$ is finite. With a similar hypothesis, Ellis in \cite{Ellis2001} gave bounds on the order and exponent of $\gamma_{n+1}(H)$.  Recently the authors of \cite{Donadze2020, Donadze21} provided generalizations of Baer's theorem to other classes of groups. Theorems \ref{MR: H/Z,H/D powerful gamma_n,P_n,Gamma_n powerful} and \ref{MR: extension of powerful implies lower central and derived are powerful} can be considered as generalizations of Baer's theorem. 



\begin{outline}

    Section \ref{section: preliminary results} collects  some preliminary results. In Section \ref{section: powerful lower p, lower central, derived}, we prove our main results on powerfully embedded subgroups and $p$-adic analyticity.  Sections \ref{section: crossed square I} and \ref{section: crossed square II} develop the crossed module framework and establish results on tensor products of powerful $p$-groups. In Section~\ref{section: crossed square morphisms}, we study morphisms between crossed squares and obtain an exact sequence given in Proposition~\ref{nfold_qtensor_exact_sequence} involving $n$-fold tensor products. In Section \ref{section: finiteness of lower p series}, we study the finiteness of the lower $p$-series for a pro-$p$ group. We prove the powerfulness and finiteness of the Frattini series for a pro-$p$ group in Section \ref{frattini series}.  In Section 
    \ref{section: General version of commutator and Frattini are powerful}, we prove that generalizations of the commutator and the Frattini subgroups of $G$ are powerful. Furthermore, we prove that $\operatorname{Inn}(G)$ is powerfully embedded in $\operatorname{Aut}_\Phi(G)$, the class preserving automorphisms, and $\operatorname{IA}(G)$.
    
    As a consequence of Theorem \ref{MR: M tensor G and M ptensor G are powerful}, we obtain Theorem \ref{MR: extension of powerful implies lower central and derived are powerful}.
We obtain Theorem \ref{MR: H/Z,H/D powerful gamma_n,P_n,Gamma_n powerful} by combining  Theorems \ref{gamma_npowerful}, \ref{P_npowerful} and \ref{Gamma_n is powerful}; Theorem \ref{MR: foreven H/Z,H/D powerful gamma_n,P_n,Gamma_n powerful} by combining Theorems \ref{gamma_npowerful} and \ref{P_npowerfulevenprime}; Theorem \ref{MR: extension of powerful implies lower central and derived are powerful} is proved in Theorem \ref{gamma and derived is finite powerful}; Theorem \ref{MR: M tensor G and M ptensor G are powerful} is the combination of Theorems \ref{T:Tensor of poweful crossed module} and \ref{qtensorpowerful}. By combining Lemma \ref{InnG powerfully embedded in Aut_Phi(G)}, Lemma \ref{Inn G pe}, and Corollary 
\ref{Corollary: Generalizations of Commutator and Frattini are powerful},
we obtain Theorem \ref{MR: Inn G pe Automorphism subgroups and L_A and Phi_A powerful}.

\end{outline}

\begin{Notations} We will denote the commutator subgroup of $H$ as $\gamma_2(H)$; $\gamma_n(H)$ will denote the $n$th term of the lower central series of $H$; $\Gamma_n(H)$ will denote the $n$th term of the derived series of $H$; $Z(H)$ will denote the center of the group; $Z_n(H)$ will denote the $n$th center of $H$; $\Phi_2(H)$ will denote the Frattini subgroup of $H$; $\Phi_n(H)$ denotes the $n$th term of the Frattini series (cf. Section \ref{frattini series}); $P_n(H)$ will denote the $n$th term of the lower $p$-series of a pro-$p$ group $H$, and  $\lambda_n(H)$ will denote the $n$th term of the lower $p$-series of an arbitrary group $H$; $\im{f}$ denotes the image of $f$.
\end{Notations}
\section{Preliminary results} \label{section: preliminary results}

In this section, we collect some preliminary results that will be needed in later sections.
\begin{lemma}\cite[Exercise 1.3, Corollary 1.3.12]{McK2000}\label{PR:R:1}
     Let $L, M$ and $N$ are normal subgroups of a group $G$, then 
     \begin{enumerate}
         \item[(i)] $[LM,N]=[L,N][M,N].$ 
         \item[(ii)]If $k\geq j\geq 1$ then $[\gamma_j(G),Z_k(G)]\leq Z_{k-j}(G)$. 
         \item[(iii)] $[N,M]=[M,N]$.
     \end{enumerate}
 \end{lemma}

\begin{lemma}\label{PR:R:2}
Let $G$ be a topological group and $M,N \leqslant G$ \begin{enumerate}
    \item[(i)] The function $\phi:G\times G\to G$ defined as $\phi(g,h)= [g,h]$ and the function $\psi:G\to G$ defined as $\psi(g)= g^p$ are continuous and hence $\big({\overline{M}}\big)^p \leqslant \overline{M^p}$ and $[\overline{M},\overline{N}] \leqslant \overline{[M,N]}$.
    \item[(ii)]  $\overline{M}\;\overline{N} \leqslant \overline{MN}$ and the equality holds if $M$ is compact and $N$ is closed. 
\end{enumerate} 
 \end{lemma}

 \begin{lemma}\label{Commutator comparision}
Let $A$, $B$, $C$ and $D$ be normal subgroup of a group $G$. If $A\leqslant B$ and $C\leqslant D$, then $[A,C]\leqslant [B,D].$ 
\end{lemma}
\begin{theorem} \cite[Theorems 1.1, 1.3, Corollary 1.2, Proposition 1.7]{LM87} \label{MannLubotzky_powerfully_embedded}
    Let $G$ be a finite $p$-group and let $M,N \unlhd G$. The following statements hold
    \begin{enumerate}
        \item[(i)] Let $M$ and $N$ be powerfully embedded in $G$. Then $[N,G]$, $N^p$, $MN$, and $[M,N]$ are powerfully embedded in $G$.
        \item[(ii)] Let $G$ be powerful. The subgroups $\gamma_i(G)$, $Z_i(G)$, $G^{p^i}$ and $P_i(G)$ are powerfully embedded in $G$.
        \item[(iii)] Let $G$ be powerful. Then each element of $G^{p^i}$ can be written as $a^{p^i}$ for some $a \in G$.
    \end{enumerate} 
\end{theorem}

\begin{theorem}[{\cite[Theorem 3.6\textit{(iii)}]{DDMS99}}] \label{MannLubotzky_powerfully_embeddedforpropgroups}
    Let $G$ be a finitely generated powerful pro-$p$ group. For each $i \in \mathbb{N}$, $G^{p^i}=\{x^{p^i} \mid x \in G\}$.
\end{theorem}

 \begin{theorem} [{\cite[Theorem 1.4]{NS2007}}]\label{NikSeg2007}
     Let $G$ be a finitely generated profinite group and $H$ a closed normal subgroup of $G$. Then the subgroup $[H,G]$, generated (algebraically) by all commutators $[h,g]$ $(h \in H, g \in G)$, is closed in $G$. In particular, each term of the lower central series $\gamma_n(G)$ is a closed subgroup of $G$.
 \end{theorem}


 \begin{theorem}[{\cite[Theorem 2.4]{AlSaAn2008}}]\label{PR:T:1} 
 Let $G$ be a pro-$p$ group and let $N$ and $M$ be closed normal subgroups of $G$, then \[[N^{p^k},M] \equiv [N,M]^{p^k} (\text{mod } [M,\prescript{}{p}{N}]^{p^{k-1}} [M,\prescript{}{p^2}{N}]^{p^{k-2}} \ldots [M,\prescript{}{p^k}{N}])\]
 \end{theorem}

  \begin{theorem}[{\cite[Theorem 2.5]{AlSaAn2008}}]\label{PR:T:2} 
 Let $G$ be a pro-$p$ group and let $N$ be a closed normal subgroup of $G$, then the following congruence hold for every $k,l \geq 0$ \[[N^{p^k},\prescript{}{l}{G}] \equiv [N,\prescript{}{l}{G}]^{p^k} (\text{mod } \prod_{r=1}^k [N^{p^{k-r}},\prescript{}{r(p-1)+l}{G}])\]
 \end{theorem}

\section{Powerfulness of the terms of the lower p-series, the lower central series and the derived series} \label{section: powerful lower p, lower central, derived}
The following proposition reduces the study of powerful embeddings in a pro-$p$ group to its finite quotients.
\begin{prop} {\cite[Proposition 3.2]{DDMS99}}\label{Dixon_Profinite_to_finite}
    Let $G$ be a group pro-$p$ and $N$ be an open subgroup of $G$. Then $N$ is powerfully embedded in $G$
if and only if $NK/K$ is powerfully embedded in $G/K$ for every open normal subgroup $K$ of $G$.
\end{prop}
The above proposition also holds true for a closed subgroup $N$ of $G$.
\begin{theorem} \label{gamma_npowerful}
    Let $p$ be a prime and $H$ be a pro-$p$ group. If $H/Z_{n-1}(H)$ is powerful, then $\overline{\gamma_n(H)}$ is powerfully embedded in $H$.
\end{theorem}

\begin{proof}
    Note that $(H/K)/(Z_{n-1}(H/K))$ is a quotient of the group $H/Z_{n-1}(H)$ and image of $\overline{\gamma_n(H)}$ in $H/K$ is $\gamma_n(H/K)$, for every open normal subgroup $K$ of $H$. Hence, by Proposition \ref{Dixon_Profinite_to_finite}, it suffices to prove the statement for a finite $p$-group $H$. We will prove the theorem for an odd prime. The proof for $p=2$ follows \textit{mutatis mutandis}. Since $H/Z_{n-1}(H)$ is powerful, we have that $\gamma_2(H) \leqslant H^p Z_{n-1}(H)$. Observe that $[\gamma_n(H),H]=[\gamma_2(H), \prescript{}{n-1}{H}]$. Therefore,
    \begin{align}
        [\gamma_n(H),H] &\leqslant [H^p Z_{n-1}(H), \prescript{}{n-1}{H}]\nonumber\\
        & \leqslant [H^p, \prescript{}{n-1}{H}][ Z_{n-1}(H), \prescript{}{n-1}{H}]&&\mbox{(Lemma \ref{PR:R:1}\textit{(i)})}\nonumber\\
        &=[H^p, \prescript{}{n-1}{H}]. &&\mbox{(Lemma \ref{PR:R:1}\textit{(ii)})}\label{Central series step 1}
    \end{align}Using Theorem \ref{PR:T:2}, we obtain
    \begin{align}
        [H^p, \prescript{}{n-1}{H}]&\leqslant[H, \prescript{}{n-1}{H}]^p[H, \prescript{}{p-1+n-1}{H}]\nonumber\\
        &=\gamma_n(H)^p\gamma_{n+p-1}(H).\label{central series step 2}
    \end{align} Using \eqref{Central series step 1} and \eqref{central series step 2}, and noting that $p$ is an odd prime, we obtain that {$[\gamma_n(H),H] \leqslant \gamma_n(H)^p\gamma_{n+2}(H)$}. We now show that $\gamma_{n+k}(H) \leqslant \gamma_n(H)^p\gamma_{n+k+1}(H)$ for $k \geq 1$. Towards that end,
    \begin{align}
		\gamma_{n+k}(H)&=[\gamma_2(H), \prescript{}{n+k-2}{H}] \nonumber\\
		&\leqslant [H^pZ_{n-1}(H), \prescript{}{n+k-2}{H}] \nonumber\\
            &\leqslant [H^p, \prescript{}{n+k-2}{H}][Z_{n-1}(H), \prescript{}{n+k-2}{H}]\nonumber\\
            &= [H^p, \prescript{}{n+k-2}{H}]\nonumber\\
            &\leqslant [H, \prescript{}{n+k-2}{H}]^p[H, \prescript{}{p-1+n+k-2}{H}] && \mbox{(Theorem } \ref{PR:T:2})\nonumber\\
            &= \gamma_{n+k-1}(H)^p\gamma_{n+k+p-2}(H)\nonumber\\
            &\leqslant \gamma_n(H)^p\gamma_{n+k+1}(H).\label{central series step 3}
    \end{align}
    Using \eqref{central series step 3}, we obtain that $[\gamma_n(H),H] \leqslant \gamma_n(H)^p$ and hence the proof.
\end{proof}

\begin{Remark}    
    If we define the notion of powerfully embedded group for any subgroup $N$ of a pro-$p$ group $H$ as $[H,H] \leqslant H^p$, then the closure in the statement of the Theorem \ref{Dixon_Profinite_to_finite} can be removed.
\end{Remark}
      \noindent If $G$ is a finitely generated pro-$p$ group, then $\gamma_n(G)$ is closed by Theorem~\ref{NikSeg2007}, and hence we have the following corollary.
\begin{corollary}
        Let $p$ be a prime and $H$ be a finitely generated pro-$p$ group. If $H/Z_{n-1}(H)$ is powerful, then $\gamma_n(H)$ is powerfully embedded in $H$.
\end{corollary}

\begin{theorem} \label{P_npowerful}
    Let $p$ be an odd prime and $H$ be a pro-$p$ group. If $H/Z_{n-1}(H)$ is powerful, then $P_n(H)$ is powerfully embedded in $H$.
\end{theorem}

\begin{proof}
As before, we may assume that $H$ is a finite $p$-group. Recall that $P_n(H)=\prod \limits _{i=1}^n \gamma_i(H)^{p^{n-i}}$, and therefore 
\begin{align}\label{Lazard theorem step -2}
    [P_n(H),H]=\prod \limits _{i=1}^n [\gamma_i(H)^{p^{n-i}},H].
\end{align} For $i < n$, using Theorem \ref{PR:T:1}, we observe that 
\begin{align}\label{Lazard theorem step -1}
    [\gamma_i(H)^{p^{n-i}},H]  \leqslant[\gamma_i(H),H]^{p^{n-i}} \prod \limits _{j=1}^{n-i}[H, \prescript{}{p^{j}}{\gamma_i(H)}]^{p^{n-i-j}}.
\end{align}

\begin{align}
    [H, \prescript{}{p^{j}}{\gamma_i(H)}]^{p^{n-i-j}} &\leqslant \gamma_{ip^j+1}(H)^{p^{n-i-j}}\nonumber \\ 
    &\leqslant \gamma_{i+j+1}(H)^{p^{n-(i+j+1)+1}} && (ip^j+1 \geq i+j+1 \text{ for } p \geq 2)\label{Lazard theorem step 0}.
\end{align}
For $j$ with $1 \leq j < n-i$, we have that $i+j+1 \leq n$. Then 
\begin{align}
    \gamma_{i+j+1}(H)^{p^{n-(i+j+1)+1}} &\leqslant \prod \limits _{k=1}^n \gamma_k(H)^{p^{n-k+1}}\nonumber\\
    &\leqslant\prod \limits _{k=1}^n (\gamma_k(H)^{p^{n-k}})^p\leqslant(P_n(H))^p.\label{Lazard theorem step 1}
\end{align}
For $j=n-i$, 
\begin{align}
    [H, \prescript{}{p^{n-i}}{\gamma_i(H)}]^{p^{n-i-(n-i)}} \leqslant \gamma_{i+(n-i)+1}(H) = \gamma_{n+1}(H)\label{lazard theorem step 2}
\end{align}
Using \eqref{Lazard theorem step 0}, \eqref{Lazard theorem step 1} and \eqref{lazard theorem step 2}, we obtain 
\begin{align}\label{Lazard theorem step 3}
    \prod \limits _{j=1}^{n-i}[H, \prescript{}{p^{j}}{\gamma_i(H)}]^{p^{n-i-j}}\leqslant (P_n(H))^p\gamma_{n+1}(H).
\end{align}\

Note that \begin{align}\label{Lazard theorem step 4}
    [\gamma_i(H),H]^{p^{n-i}} = \gamma_{i+1}(H)^{p^{n-(i+1)+1}}
    \leqslant(\gamma_{i+1}(H)^{p^{n-(i+1)}})^p\leqslant(P_n(H))^p.
\end{align}

Using \eqref{Lazard theorem step 3} and \eqref{Lazard theorem step 4} in \eqref{Lazard theorem step -1}, we obtain

\begin{align}\label{Lazard theorem step 5}
    [\gamma_i(H)^{p^{n-i}},H] \leqslant (P_n(H))^p\gamma_{n+1}(H),
\end{align} for all $i<n.$ Moreover, for $i=n$, we have \begin{align}\label{Lazard theorem step 6}
    [\gamma_i(H)^{p^{n-i}},H]=\gamma_{n+1}(H).
\end{align}
Using \eqref{Lazard theorem step 5} and \eqref{Lazard theorem step 6} in \eqref{Lazard theorem step -2}, we get ${[P_n(H),H] \leqslant (P_n(H))^p \gamma_{n+1}(H)}$. Since $H/Z_{n-1}(H)$ is powerful, using \eqref{central series step 3} (this does not require $H$ to be finitely generated), we obtain \begin{align}\label{Lazard theorem step 7}
    \gamma_{n+k} \leqslant \gamma_n(H)^p\gamma_{n+k+1}(H), \text{ for all $k\geq1$}.
\end{align}
Observing that $\gamma_n(H)^p\leqslant P_n(H)^p$ and using \eqref{Lazard theorem step 7} repeatedly, we obtain $[P_n(H),H]  \leqslant  P_n(H)^p$, and hence the proof.
\end{proof}

\begin{theorem} \label{P_npowerfulevenprime}
    Let $p=2$ and $H$ be a pro-$p$ group. If $H/Z_{n-1}(H)$ is powerful, then $P_n(H)$ is powerful.
\end{theorem}
\begin{proof}
    The proof proceeds along the same lines. Although additional terms appear and some extra technical work is required, the argument can be adapted to yield the above theorem.
\end{proof}

As a corollary, we obtain the following elegant result.


\begin{corollary} \label{padic analytic from padic analytic}
    Let $p$ be a prime, and $H$ be a finitely generated pro-$p$ group. If $H/Z_n(H)$ is $p$-adic analytic for some $n\in \mathbb{N}$, then so is $H$.
\end{corollary}

\begin{proof}
    By Theorem \ref{padic analytic characterization} (see also \cite[III, 3.4.3]{La1965}, \cite[Theorem 2.1]{LM872}), a finitely generated pro-$p$ group is analytic if and only if it has a powerful subgroup of finite index. Therefore, there exists a subgroup $N/Z_n(H)$ which is a powerful subgroup of finite index. We have that $N/Z_n(N)$ is a quotient of $N/Z_n(H)$ and hence $N/Z_n(N)$ is powerful. By Theorem \ref{P_npowerful} and Theorem \ref{P_npowerfulevenprime}, we have that $P_{n+1}(N)$ is powerful. Since $P_{n+1}(N)$ is open in $N$ and $N$ is open in $H$, we have that $H$ is also $p$-adic analytic.
\end{proof}


We define $\Gamma_1(H)=H$ and $\Gamma_{n+1}(H)=\overline{[\Gamma_n(H),\Gamma_n(H)]}$ for all $n\geq 1.$ While the derived series is traditionally defined without taking closures, we include the closure here to ensure that it remains a pro-$p$ group. This adjustment aligns with the definition of powerfully embedded subgroup, which is defined for closed normal subgroups.
\begin{lemma}\label{Lemma Derived series 1}
    Let $H$ be a profinite group, $p \geq 3$ and $k \geq 3$
    be natural numbers. Then, \[\bigr[\Gamma_i,_{\;p}[[H,_{\;k-1}H],\Gamma_2,\ldots, \Gamma_{i-1}]\bigl]\leqslant \bigr[[H,_{\;p+k-2}H],\Gamma_2,\ldots,\Gamma_i\bigr]\] for all natural numbers $i\geq 3$.
\end{lemma}
    \begin{proof} Using Lemma \ref{PR:R:1}\textit{(iii)}, we obtain
        \begin{align*}
            \bigr[\Gamma_i,_{\;p}&[[H,_{\;k-1}H],\Gamma_2,\dots ,\Gamma_{i-1}]\bigr]\\
            &=\bigl[[[H,_{\;k-1}H],\Gamma_2,\dots,\Gamma_{i-1}],\Gamma_i,_{\;p-1}[[H,_{\;k-1}H],\Gamma_2,\dots,\Gamma_{i-1}]\bigr]\\
            &\leqslant \bigl[[[H,_{\;k-1}H],\Gamma_2,\dots,\Gamma_{i-1}],_{\;p}\Gamma_i\bigr].
        \end{align*}Now the result follows from Lemma \ref{Commutator comparision}.
    \end{proof}
\begin{lemma}\label{Lemma Derived series 1A}
    Let $H$ be a profinite group, $p$ and $k\geq 2$ be natural numbers. Then, \[[\Gamma_k,_{\;p}\Gamma_k]\leqslant \overline{\bigl[[H,_{\;p}H],\Gamma_2,\dots,\Gamma_k]}.\]
\end{lemma}
\begin{proof}
    Note that $[\Gamma_k,_{\;p}\Gamma_k] \leqslant\overline{[\Gamma_1,\Gamma_1,\Gamma_2,\dots,\Gamma_{k-1},_{\;p}\Gamma_k]}$ by Lemma \ref{PR:R:2}\textit{(i)}. Now the proof follows by Lemma \ref{Commutator comparision}.
\end{proof}

\begin{lemma}\label{Lemma Derived series 2}
Let $p$ be an odd prime, and $n\geq 3$, $k\geq 3$ be natural numbers. If $H$ is a finitely generated pro-$p$-group, then 
 \begin{align*}
     \bigl[[H,_{\;k-1}H]^p,\Gamma_2,\ldots,\Gamma_{n-1}\bigr]
        \leqslant \overline{(\Gamma_n)^p\bigl[[H,_{\;p+k-2}H],\Gamma_2,\ldots,\Gamma_{n-1}\bigr]}.
 \end{align*}
    
\end{lemma}
\begin{proof}
    Using Theorem \ref{PR:T:1} and Lemma \ref{PR:R:1}\textit{(i)}, we obtain 
    \begin{align*}
        &\bigl[[H,_{\;k-1}H]^p,\Gamma_2,\Gamma_3,\dots,\Gamma_{n-1}\bigr]\\
        &\leqslant  \bigl[[[H,_{\;k-1}H],\Gamma_2]^p,\Gamma_3,\dots,\Gamma_{n-1}\bigr]\bigl[[\Gamma_2, _{\;p}[H,_{\;k-1}H]],\Gamma_3,\dots,\Gamma_{n-1}\bigr]\nonumber\\
        &\leqslant \bigl[[[H,_{\;k-1}H],\Gamma_2]^p,\Gamma_3,\dots,\Gamma_{n-1}\bigr] \bigl[[\Gamma_2, _{\;p-1}[H,_{\;k-1}H], [H,_{\;k-1}H]],\Gamma_3,\dots,\Gamma_{n-1}\bigr].    
    \end{align*}
     Observe that \[\bigl[[\Gamma_2, _{\;p-1}[H,_{\;k-1}H], [H,_{\;k-1}H]],\Gamma_3,\dots,\Gamma_{n-1}\bigr]\leqslant \bigl[[\gamma_{2+(p-1)k}(H),\Gamma_2],\Gamma_3,\dots,\Gamma_{n-1}\bigr].\]
     Since $p+k-2 \leq 2+(p-1)k$, 
     \begin{align*}
     [\gamma_{2+(p-1)k}(H),\Gamma_2,\Gamma_3,\dots,\Gamma_{n-1}]\leqslant \bigl[[H,_{\;p+k-2}H],\Gamma_2,\ldots,\Gamma_{n-1}\bigr].
     \end{align*}Thus,
     \begin{align}
        \bigl[[H,&_{\;k-1}H]^p,\Gamma_2,\Gamma_3,\dots,\Gamma_{n-1}\bigr]\nonumber\\
        &\leqslant  \bigl[[[H,_{\;k-1}H],\Gamma_2]^p,\Gamma_3,\dots,\Gamma_{n-1}\bigr]\bigl[[H,_{\;p+k-2}H],\Gamma_2,\ldots,\Gamma_{n-1}\bigr].\label{Derived series Lemma step 1}\end{align}
 
        Let $i$ be a natural number satisfying $2 \leq i \leq n-2$. Using Theorem \ref{PR:T:1}, Lemma \ref{PR:R:2} and Lemma \ref{Lemma Derived series 1} to   yields
        \begin{align} 
        \bigl[[&[H,_{\;k-1}H],\Gamma_2,\dots, \Gamma_i]^p,\Gamma_{i+1}\bigr]\nonumber\\
        &\leqslant\bigl[\overline{[[H,\prescript{}{k-1}{H}],\Gamma_2,\dots, \Gamma_i]}^p,\Gamma_{i+1}\bigr]\nonumber\\
            &\leqslant \bigl[\overline{[[H,_{\;k-1}H],\Gamma_2,\dots, \Gamma_i]},\Gamma_{i+1}\bigr]^p\bigr[\Gamma_{i+1},_{\;p}\overline{[[H,_{\;k-1}H],\Gamma_2,\dots, \Gamma_{i}]}\bigl]\nonumber\\
            & \leqslant \overline{\bigl[[[H,_{\;k-1}H],\Gamma_2,\dots, \Gamma_i],\Gamma_{i+1}\bigr]^p\bigr[\Gamma_{i+1},_{\;p}[[H,_{\;k-1}H],\Gamma_2,\dots, \Gamma_{i}]\bigl]} \nonumber\\
            &\leqslant \overline{\bigl[[H,_{\;k-1}H],\Gamma_2,\dots, \Gamma_{i+1}\bigr]^p\bigr[[H,_{\;p+k-2}H],\Gamma_2,\dots,\Gamma_{i+1}\bigr]}.\label{Derived series Zapparin step}
        \end{align} 
        Using \eqref{Derived series Lemma step 1}, \eqref{Derived series Zapparin step}, Lemma \ref{PR:R:2}\textit{(ii)} and Lemma \ref{PR:R:1}\textit{(i)}, we obtain
        \begin{align*}
        \bigl[[H,&_{\;k-1}H]^p,\Gamma_2,\Gamma_3,\ldots,\Gamma_{n-1}\bigr]\\
        &\leqslant  \overline{\bigl[[[H,_{\;k-1}H],\Gamma_2,\ldots, \Gamma_{n-2}],\Gamma_{n-1}\bigr]^p\bigl[[H,_{\;p+k-2}H],\Gamma_2,\ldots,\Gamma_{n-1}\bigr]}\\
        &\leqslant \overline{\bigl[[[H,H],\Gamma_2,\ldots, \Gamma_{n-2}],\Gamma_{n-1}\bigr]^p\bigl[[H,_{\;p+k-2}H],\Gamma_2,\ldots,\Gamma_{n-1}\bigr]}\\
        &\leqslant \overline{(\Gamma_n)^p\bigl[[H,_{\;p+k-2}H],\Gamma_2,\ldots,\Gamma_{n-1}\bigr]}
        \end{align*}and hence the proof.
\end{proof}

\begin{definition}\cite[Definition 3.6]{Donadze21}\label{D_n(G)}
    Let $H$ be a group. For each $n\geq 1$, we define a group $\mathcal{D}_n(H)$ by\[\mathcal{D}_n(H)=\{h\in H\mid[\ldots[[h,x_1],x_2],\ldots,x_n]=1 \text{ for each }x_i\in \Gamma_i(H)\}.\]
\end{definition}

Note that $\mathcal{D}_n(H)$ is a closed normal subgroup of $H$.

\begin{theorem} \label{Gamma_n is powerful}
    Let $p$ be an odd prime and $H$ be a pro-$p$ group. If $H/\mathcal{D}_{n-1}(H)$ is powerful for some natural number $n\geq 2$, then $\overline{\Gamma_n(H)}$ is powerfully embedded in $H$.
\end{theorem}
\begin{proof}
    As before, we can assume that $H$ is finite. The case $n=2$ follows from Theorem \ref{gamma_npowerful}. For ease of notation, we will denote $\Gamma_i(H)$ by $\Gamma_i$ throughout this proof. Now we first show that ${[\Gamma_n,H]\leqslant \bigl[[\Gamma_2,H], \Gamma_2,\ldots, \Gamma_{n-1}\bigr]}$ for any $n\geq 3.$ By the three subgroup lemma, the case $n=3$ follows easily. Assume the result is true for $n=k-1$, that is, $[\Gamma_{k-1},H]\leqslant \bigl[[\Gamma_2,H], \Gamma_2,\ldots, \Gamma_{k-2}\bigr]$. By the three subgroup lemma, we have 
    \begin{align*}
        [\Gamma_k, H]&= \bigl[[\Gamma_{k-1},\Gamma_{k-1}],H\bigr]\\
        &\leqslant \bigl[[\Gamma_{k-1}, H],\Gamma_{k-1}\bigr]\leqslant \bigl[[\Gamma_2,H], \Gamma_2,\ldots, \Gamma_{k-2},\Gamma_{k-1}\bigr],
    \end{align*}and hence by induction on $n$ the claim follows. Since $H/\mathcal{D}_{n-1}(H)$ is powerful, we have $\Gamma_2\leqslant H^p\mathcal{D}_{n-1}(H)$. Thus,
    \begin{align*}
        [\Gamma_n, H]&\leqslant \bigl[[\Gamma_2,H], \Gamma_2,\ldots, \Gamma_{n-1}\bigr]\\
        &\leqslant \bigl[[H^p\mathcal{D}_{n-1}(H),H], \Gamma_2,\ldots, \Gamma_{n-1}\bigr]\\
        &\leqslant \bigl[[H^p,H], \Gamma_2,\ldots, \Gamma_{n-1}\bigr]\bigl[[\mathcal{D}_{n-1}(H),H], \Gamma_2,\ldots, \Gamma_{n-1}\bigr].
    \end{align*}By definition of $\mathcal{D}_{n-1}(H)$, we have $\bigl[[\mathcal{D}_{n-1}(H),H], \Gamma_2,\ldots, \Gamma_{n-1}\bigr]=1$. Using Theorem \ref{PR:T:1}, we obtain 
    \begin{equation}
         [\Gamma_n, H]\leqslant\bigl[[H,H]^p,\Gamma_2,\ldots, \Gamma_{n-1}\bigr]\bigl[[H,_{\;p} H],\Gamma_2,\ldots, \Gamma_{n-1}\bigr].\label{Derived series claim 0}
    \end{equation}
  
      We now show by induction on $i$ that 
      \begin{align}
          [\Gamma_n, H]\leqslant \bigl[(\Gamma_i)^p,\Gamma_i,\ldots, \Gamma_{n-1}\bigr]\bigl[[H,_{\;p} H],\Gamma_2,\ldots, \Gamma_{n-1}\bigr] \label{Theorem derived powerful_claim 1}
      \end{align}
      for any $i$ with $2\leq i<n-1$. The case $i=2$ follows from \eqref{Derived series claim 0}. Assume the claim is true for $i=k$, and applying Theorem \ref{PR:T:1}, Lemma \ref{PR:R:1}\textit{(i)} and Lemma~\ref{Lemma Derived series 1A} yields
      \begin{align*}
         [&\Gamma_n, H]\\
         &\leqslant \bigl[[(\Gamma_k)^p,\Gamma_k],\Gamma_{k+1},\ldots, \Gamma_{n-1}\bigr]\bigl[[H,_{\;p} H],\Gamma_2,\ldots, \Gamma_{n-1}\bigr] \\
         &\leqslant  \bigl[[\Gamma_k,\Gamma_k]^p,\Gamma_{k+1},\ldots, \Gamma_{n-1}\bigr]\bigl[[\Gamma_k,_{\;p}\Gamma_k],\Gamma_{k+1},\ldots, \Gamma_{n-1}\bigr]\bigl[[H,_{\;p} H],\Gamma_2,\ldots, \Gamma_{n-1}\bigr]\\
         &\leqslant  \bigl[\big( [\Gamma_k,\Gamma_k]\big)^p,\Gamma_{k+1},\ldots, \Gamma_{n-1}\bigr] \bigl[[H,_{\;p} H],\Gamma_2,\ldots, \Gamma_{n-1}\bigr]\bigl[[H,_{\;p} H],\Gamma_2,\ldots, \Gamma_{n-1}\bigr]\\
         &\leqslant  \bigl[(\Gamma_{k+1})^p,\Gamma_{k+1},\ldots, \Gamma_{n-1}\bigr]\bigl[[H,_{\;p} H],\Gamma_2,\ldots, \Gamma_{n-1}\bigr].
      \end{align*}
        Hence \eqref{Theorem derived powerful_claim 1} holds. Similarly, we have that 
            \begin{align}
         [\Gamma_n, H]&\leqslant  \bigl[[(\Gamma_{n-2})^p,\Gamma_{n-2}],\Gamma_{n-1}\bigr]\bigl[[H,_{\;p} H],\Gamma_2,\ldots, \Gamma_{n-1}\bigr]  \nonumber\\
         &\leqslant  \bigl[[\Gamma_{n-2},\Gamma_{n-2}]^p,\Gamma_{n-1}\bigr]\bigl[[\Gamma_{n-2},_{\;p}\Gamma_{n-2}],\Gamma_{n-1}\bigr]\bigl[[H,_{\;p} H],\Gamma_2,\ldots, \Gamma_{n-1}\bigr]\nonumber\\
         &\leqslant  \bigl[(\Gamma_{n-1})^p,\Gamma_{n-1}\bigr]\bigl[[H,_{\;p} H],\Gamma_2,\ldots, \Gamma_{n-1}\bigr] \nonumber\\
         &\leqslant  \bigl[\Gamma_{n-1},\Gamma_{n-1}\bigr]^p\bigl[\Gamma_{n-1},_{\;p}\Gamma_{n-1}\bigr]\bigl[[H,_{\;p} H],\Gamma_2,\ldots, \Gamma_{n-1}\bigr]\nonumber\\
         &\leqslant  (\Gamma_n)^p\bigl[[H,_{\;p} H],\Gamma_2,\ldots, \Gamma_{n-1}\bigr] \label{Theorem derived powerful_claim 2}.
      \end{align}
Moreover for any $k\geq 3$, \begin{align*}
    \bigl[[H,_{\;k} H],\Gamma_2,\ldots, \Gamma_{n-1}\bigr]&= \bigl[[\Gamma_2,_{\;k-1} H],\Gamma_2,\ldots, \Gamma_{n-1}\bigr]\\
    &\leqslant\bigl[[H^p\mathcal{D}_{n-1},_{\;k-1}H],\Gamma_2,\ldots,\Gamma_{n-1}\bigr]\\
    &\leqslant\bigl[[H^p,_{\;k-1}H],\Gamma_2,\ldots,\Gamma_{n-1}\bigr].
    \end{align*}
    Applying Theorem \ref{PR:T:2} and Lemma \ref{PR:R:1}\textit{(i)} to $\bigl[[H^p,_{\;k-1}H],\Gamma_2,\ldots,\Gamma_{n-1}\bigr]$, we obtain
    \begin{align}
        \bigl[[H,&_{\;k} H],\Gamma_2,\ldots, \Gamma_{n-1}\bigr]\nonumber\\
        &\leqslant\bigl[[H,_{\;k-1}H]^p[H,_{\;p+k-2}H],\Gamma_2,\ldots,\Gamma_{n-1}\bigr]\nonumber\\
        &\leqslant\bigl[[H,_{\;k-1}H]^p,\Gamma_2,\ldots,\Gamma_{n-1}\bigr]\bigl[[H,_{\;p+k-2}H],\Gamma_2,\ldots,\Gamma_{n-1}\bigr].\label{Derived series Claim 3}
    \end{align}Using Lemma \ref{Lemma Derived series 2} in \eqref{Derived series Claim 3} yields
\begin{align}
    \bigl[[H,_{\;k}& H],\Gamma_2,\ldots, \Gamma_{n-1}\bigr]\nonumber\\
    &\leqslant   (\Gamma_n)^p\bigl[[H,_{\;p+k-2}H],\Gamma_2,\ldots,\Gamma_{n-1}\bigr]\bigl[[H,_{\;p+k-2}H],\Gamma_2,\ldots,\Gamma_{n-1}\bigr]\nonumber\\
    &\leqslant   (\Gamma_n)^p\bigl[[H,_{\;p+k-2}H],\Gamma_2,\ldots,\Gamma_{n-1}\bigr].\label{Derived series final claim}
\end{align}Repeatedly applying \eqref{Derived series final claim} in \eqref{Theorem derived powerful_claim 2} yields
\[ [\Gamma_n, H]\leqslant  (\Gamma_n)^p \bigl[[H,_{\;p+r(p-2)} H],\Gamma_2,\ldots, \Gamma_{n-1}\bigr],\] for every $r\geq 1$. Since $H$ is finite, we have $[\Gamma_n, H]\leqslant \Gamma_n^p$
and hence the proof.
\end{proof}

\section{Crossed square of powerful $p$-groups \Romannum{1}.} \label{section: crossed square I}
Recall that a crossed module is a group homomorphism $\mu: M \to G$ together with an action of $G$ on $M$, which satisfies the following conditions
\begin{align}
    &\mu(^g m)=g\mu(m)g^{-1} , g\in G, m\in M\label{E:Crossed module 1}\\
        &^{\mu(m)}m'=mm'm^{-1}, m,m'\in M\label{E:Crossed module 2}.
\end{align}
Throughout the rest of the paper, we consider only actions via automorphisms.
\begin{lemma} \label{Ginvariant is normal}
    Let $\mu: M \to G$ be a crossed module. If $N$ is a $G$-invariant subgroup of $M$, then $N$ is a normal subgroup of $M$. Moreover, the restriction of $\mu$ to $N$ is a crossed module.
\end{lemma}
\begin{proof}
    For all $m \in M$, $n \in N$, we have that $mnm^{-1}=\prescript{\mu(m)}{}{n}$. As $N$ is $G$-invariant, we get that $\prescript{m}{}{n} \in N$.
\end{proof}
	
\begin{definition} \label{crosseqsquaredefinition}
    A crossed square is a commutative square of groups
\begin{center}
    \begin{tikzcd} \label{Gtensorcrossedsquare}
	L\arrow[r,"\beta"]\arrow[d,"\alpha"] & N\arrow[d,"\nu"] 
	\\ M \arrow[r, "\mu"]& G
			\end{tikzcd}
\end{center}
together with action of $G$ on $L,M,N$ (and hence actions of $M$ on $L$ and $N$ via $\mu$ and of $N$ on $L$ and $M$ via $\nu$) and a function $h:M\times N\to L. $ This structure shall satisfy the following axioms:
\begin{enumerate}
    \item [(i)] The maps $\alpha, \beta$ preserve the action of $G$; further, with the given action, the maps $\mu,\nu $ and $\kappa=\mu\alpha=\nu\beta$ are crossed modules
    \item [(ii)]$\alpha h(m,n)=m\prescript{n}{}{m^{-1}}, \alpha h(m,n)=\prescript{m}{}{n}n^{-1}$
    \item[(iii)]$h(\alpha l,n)= l\prescript{n}{}{l^{-1}}, h(m,\beta l)=\prescript{m}{}{l}l^{-1}$
    \item [(iv)]$h(mm',n)=\prescript{m}{}{h(m',n)}h(m,n), h(m,nn')=h(m,n)\prescript{n}{}{h(m,n')}$
    \item [(v)]$h(\prescript{g}{}{m},\prescript{g}{}{n})=\prescript{g}{}{h(m,n)}$
\end{enumerate}for all $l\in L,m,m'\in M,n,n'\in N$ and $g\in G$.
\end{definition}

\begin{prop}[Proposition 2.15, \cite{BL87}]\label{tensor is crossed square}Let $\mu:M\to G, \nu:N\to G$ be a crossed modules, so that $M,N$ act on both $M$ and $N$ via $G$. Then there is a crossed square \begin{equation}
    \begin{tikzcd} \label{Diagram tensor is crossed square}
	M\otimes N\arrow[r,"\beta"]\arrow[d,"\alpha"] & N\arrow[d,"\nu"] 
	\\ M \arrow[r, "\mu"]& G
			\end{tikzcd}
\end{equation}where $\alpha(m\otimes n)=m^nm^{-1}$ and $\beta(m\otimes n)=\prescript{m}{}{n}n^{-1},$ and $h(m,n)=m\otimes n.$    
\end{prop}

Setting $N=G$ and $\nu:N\to G $ as identity in \eqref{Diagram tensor is crossed square}, the crossed module $\mu:M\to G$ induces the following crossed square
\begin{equation}\label{tenssquare}
    \begin{tikzcd} 
	M \otimes G \arrow[r,"\beta"]\arrow[d,"\alpha"] & G\arrow[d,"id"] 
	\\ M \arrow[r, "\mu"]& G
			\end{tikzcd}
\end{equation} where $\alpha$ and $\beta$ are defined in Proposition \ref{tensor is crossed square}. For $n\geq 1$, consider the natural homomorphism $\tau_n:M^{p^n}\otimes G\to M\otimes G$. By abuse of notation, we will denote $\im{\tau_n}$ as $M^{p^n}\otimes G$, and so $M^{p^n}\otimes G$  can be regarded as a subgroup of $M\otimes G$ generated by the elements of the form $m^{p^n}\otimes g$. Moreover, $M^{p^n}\otimes G$ is a $G$-invariant subgroup, and hence it is a normal subgroup of $M\otimes G$ by Lemma~\ref{Ginvariant is normal}. We now recall a result from \cite[Proposition 2.3]{BL87}, which will be used in the proof of Lemma \ref{commutator_of_ tensor_of_powerful_groups}.
\begin{lemma}[Proposition 2.3, \cite{BL87}]\label{BL identity}
    Let $M,N$ be groups equipped with compatible actions on each other.
    \begin{enumerate}
        \item [(i)] There are homomorphisms $\alpha: M\otimes N\to M$, $\beta:M\otimes N\to N$ such that $\alpha(m\otimes n)= m^nm^{-1}$, $\beta(m\otimes n)=\;^mnn^{-1}$.
        \item [(ii)] The homomorphisms $\alpha,\beta$, with the given actions, are crossed modules.
        \item [(iii)] If $l\in M\otimes N, m\in M, n\in N$, then 
        \begin{align*}
            (\alpha l)\otimes n&= l^nl^{-1}\\
            m\otimes\beta l&= \;^mll^{-1}.
        \end{align*}
        \item[(iv)] If $l,l'\in M\otimes N$, then $[l,l']=\alpha l\otimes \beta l'.$
        
    \end{enumerate}
\end{lemma}
\begin{lemma}\label{commutator_of_ tensor_of_powerful_groups}
    Let $p$ be an odd prime, $M$ and $G$ be finite $p$-groups, and ${\mu:M\to G}$ be a crossed module. If $M$ is a powerful $p$-group and $\mu(M)$ is powerfully embedded in $G$, then $\gamma_2(M\otimes G)\leqslant M^p\otimes G.$
\end{lemma}
\begin{proof}
    Let $m\otimes g, m_1\otimes g_1\in M\otimes G$, we will show that $[m\otimes g, m_1\otimes g_1]\in M^p\otimes G$. Applying Lemma \ref{BL identity}\textit{(iv)}) for the first equality, Lemma \ref{BL identity}\textit{(i)} for the second equality, and Lemma \ref{BL identity}\textit{(iii)} for the third equality, we obtain
   \begin{align*}
       [m\otimes g, m_1\otimes g_1]&
       =\alpha(m\otimes g)\otimes \beta(m_1\otimes g_1) \\
       &=m^gm^{-1}\otimes \beta(m_1\otimes g_1)\\
       &=\prescript{m^gm^{-1}}{}{(m_1\otimes g_1)}(m_1\otimes g_1)^{-1} \\
       &=(\prescript{m^gm^{-1}}{}{m_1}\otimes\; ^{m^gm^{-1}}g_1)(m_1\otimes g_1)^{-1}.
       \end{align*}
       Recall that by Definition \ref{crosseqsquaredefinition}, the action of the group $M$ on $G$ is given by $\prescript{m}{}{g}=\prescript{\mu(m)}{}{g}$.  Note that $\prescript{m}{}{m_1}=mm_1m^{-1}=m_1[m_1^{-1},m]$ for all $m,m_1 \in M$. Thus, 
       \begin{align*}
       [m\otimes g, m_1\otimes g_1]
       &=(m_1[m_1^{-1},{m^gm^{-1}}]\otimes g_1[g_1^{-1},\mu({m^gm^{-1}})])(m_1\otimes g_1)^{-1}\\
       &=\;^{m_1}([m_1^{-1},{m^gm^{-1}}]\otimes g_1[g_1^{-1},\mu({m^gm^{-1}})])\\&\qquad \qquad\quad(m_1\otimes g_1[g_1^{-1},\mu({m^gm^{-1}})])(m_1\otimes g_1)^{-1}.
   \end{align*}Since $M$ is powerful, $^{m_1}([m_1^{-1},{m^gm^{-1}}]\otimes g_1[g_1^{-1},\mu({m^gm^{-1}})])\in M^p\otimes G$. Thus,
   \begin{align*}
       [m\otimes g, m_1\otimes g_1]&\equiv (m_1\otimes g_1[g_1^{-1},\mu({m^gm^{-1}})])(m_1\otimes g_1)^{-1}\mod M^p\otimes G\\
       &\equiv (m_1\otimes g_1)^{g_1}(m_1\otimes [g_1^{-1},\mu({m^gm^{-1}})])(m_1\otimes g_1)^{-1}\mod M^p\otimes G
   \end{align*}Thus it is enough to show $m_1\otimes [g_1^{-1},\mu({m^gm^{-1}})]\in M^p\otimes G$. Since $\mu(M)$ is powerfully embedded in $G$, we have ${\mu({m^gm^{-1}})=[\mu(m),g]}\in \mu(M)^p$. Therefore, by Theorem \ref{MannLubotzky_powerfully_embedded}\textit{(iii)}, we have that ${\mu({m^gm^{-1}})=\mu(z)^p}$ for some $z\in M$. Thus,
   \begin{align*}
       m_1\otimes [g_1^{-1},\mu({m^gm^{-1}})]&= m_1\otimes[\mu(z)^p,g_1^{-1}]^{-1}\\
       &=m_1\otimes[\mu(z^p),g_1^{-1}]^{-1}\\
       &=m_1\otimes (^{\mu(z^p)}g^{-1}_1g_1)^{-1}\\
       &=m_1\otimes\beta(z^p\otimes g^{-1}_1)^{-1}\\
       &=\;^{m_1}(z^p\otimes g^{-1}_1)^{-1}(z^p\otimes g^{-1}_1).
   \end{align*}
   Therefore, $ m_1\otimes [g_1^{-1},\mu({m^gm^{-1}})]\in M^p\otimes G$, and hence the proof.
\end{proof}
\begin{lemma}[Exercise 2, p. 45, \cite{DDMS99}]\label{Shalev identity}Let $M,N\ \unlhd  G.$ If $M$ and $N$ are powerfully embedded in $G$, then $[N^{p^i}, M^{p^j}]=[N,M]^{p^{i+j}}$ for all $i$ and $j$.
    
\end{lemma}
\begin{lemma}\label{commutator in recursive formula}
    Let $p$ be an odd prime, $n$ be a nonnegative integer. Let $M$ and $G$ be finite $p$-groups, and let $\mu:M\to G$ be a crossed module. Suppose that $M$ is powerful and that $\mu(M)$ is powerfully embedded in $G$. Then, for all $m \in M^{p^n}, m_1\in M, g\in G, l\in M\otimes G$, the following holds:\[[m\otimes [\mu(m_1),g], l ]\equiv 1\mod M^{p^{n+2}}\otimes G.\] 
\end{lemma}
\begin{proof}Since $\mu(M)$ is powerfully embedded in $G$, we have $[\mu(m_1),g]=\mu(z)^p$ for some $z\in M$. Thus,\begin{align*}
        [m\otimes [\mu(m_1),g], l ]&=\alpha(m\otimes [\mu(m_1),g])\otimes \beta(l) & \mbox{(Lemma } \ref{BL identity}\textit{(iv)})\\
        &= m^ {[\mu(m_1),g]}m^{-1}\otimes \beta(l)\\
        &= m\;^{\mu(z^p)}m^{-1}\otimes \beta(l)\\
        &=mz^pm^{-1}z^{-p}\otimes \beta(l) & \mbox{(By } \eqref{E:Crossed module 2})\\
        &=[m,z^p]\otimes\beta(l).
    \end{align*}Applying Lemma \ref{Shalev identity}, we obtain $[M^{p^n},M^{p}] \leq [M,M]^{p^{n+1}}\leq M^{p^{n+2}}.$ Therefore, $[m,z^p]\otimes\beta(l)\in M^{p^{n+2}}\otimes G$, and hence the proof.
\end{proof}

\begin{prop}\label{Power_expansion_tensor_identity}
      Let $p$ be an odd prime, $n$ be a nonnegative integer. Let $M$ and $G$ be finite $p$-groups, and let $\mu:M\to G$ be a crossed module. If $M$ is powerful, and $\mu(M)$ is powerfully embedded in $G$, then for all ${m \in M^{p^n}}$, ${g \in G}$, and natural numbers $r$ and $t$, the following relation holds:
   \[m^t \otimes g\equiv(m\otimes g)^r(m^{t-r}\otimes g)(m^{t-r}\otimes [\mu(m),g])^r(m\otimes[\mu(m),g])^{\binom{r}{2}}\mod M^{p^{n+2}}\otimes G.  \]
 \end{prop}
\begin{proof}
    Let $n\geq 0, k\in \mathbb Z,$ and $m\in M^{p^n}$. We will first show that
    \begin{equation}\label{E:recursive_formula}
        m^k\otimes g\equiv (m\otimes g)(m^{k-1}\otimes g)(m^{k-1}\otimes [\mu(m),g])\mod M^{p^{n+2}}\otimes G.
    \end{equation}
    Towards that end,
    \begin{align*}
        m^k\otimes g &=\prescript{m}{}{(m^{k-1}\otimes g)} (m\otimes g)\\
        &= (m^{k-1}\otimes\;^{\mu(m)}g)(m\otimes g)\\
        &= (m^{k-1}\otimes [\mu(m),g]g)(m\otimes g)\\
        &= (m^{k-1}\otimes [\mu(m),g])\;^{[\mu(m),g]}(m^{k-1}\otimes g) (m\otimes g)\\
        &= (m^{k-1}\otimes [\mu(m),g])\;^{\beta(m\otimes g)}(m^{k-1}\otimes g) (m\otimes g).
    \end{align*} Recall that $\beta:M\otimes G\to G$ is a crossed module. Using \eqref{E:Crossed module 2}, we have $^{\beta(m\otimes g)}(m^{k-1}\otimes g)= (m\otimes g)(m^{k-1}\otimes g)(m\otimes g)^{-1}$. Thus,
    \[m^k\otimes g= (m^{k-1}\otimes [\mu(m),g])(m\otimes g)(m^{k-1}\otimes g). \]
    Now \eqref{E:recursive_formula} follows from Lemma \ref{commutator in recursive formula} and the identity $ab=[a,b]ba$. Fixing $n,t$, the proof of the Proposition
    proceeds by induction on $r$. For $r=1$, the proof  follows from \eqref{E:recursive_formula}. Assume that the proposition is true for $r=s$, and then using \eqref{E:recursive_formula}, we have \begin{align*}
       m^t \otimes g&\equiv(m\otimes g)^s(m^{t-s}\otimes g)(m^{t-s}\otimes [\mu(m),g])^s(m\otimes[\mu(m),g])^{\binom{s}{2}} \text{ mod }  M^{p^{n+2}}\otimes G \\
       &\equiv (m\otimes g)^s(m\otimes g)(m^{t-s-1}\otimes g)(m^{t-s-1}\otimes [\mu(m),g])\\
       &\qquad \big((m\otimes [\mu(m),g])(m^{t-s-1}\otimes[\mu(m),g])(m^{t-s-1}\otimes [\mu(m),[\mu(m),g]])\big)^s\\
       & \qquad (m\otimes [\mu(m),g])^{\binom{s}{2}} \mod M^{p^{n+2}}\otimes G.
    \end{align*}Now using Lemma \ref{commutator in recursive formula}, and collecting terms together, we have 
   \begin{align*}    
   m^t \otimes g&\equiv(m\otimes g)^{s+1}(m^{t-(s+1)}\otimes g)(m^{t-(s+1)}\otimes [\mu(m),g])^{s+1}\\
   &\qquad (m^{t-s-1}\otimes [\mu(m),[\mu(m),g]])^s(m\otimes[\mu(m),g])^{\binom{s+1}{2}}\mod M^{p^{n+2}}\otimes G
   \end{align*}Hence it is enough to show that ${(m^{t-s-1}\otimes [\mu(m),[\mu(m),g]])\equiv 1\mod  M^{p^{n+2}}\otimes G}$. Since $\mu(M)$ is powerfully embedded in $G$, we have $[\mu(m),g]=\mu(z)^p$ for some $z\in M$. So,
 \begin{align*}
         m^{t-s-1}\otimes [\mu(m),&[\mu(m),g]]\\
         &= m^{t-s-1}\otimes [[\mu(m),g],\mu(m)]^{-1}\\
      &=m^{t-s-1}\otimes [\mu(z^p),\mu(m)]^{-1}\\
      &=m^{t-s-1}\otimes(^{\mu(z^p)}\mu(m)\mu(m)^{-1})^{-1}\\
      &=m^{t-s-1}\otimes\beta(z^p\otimes\mu(m))^{-1}\\
      &=\;^{m^{t-s-1}}(z^p\otimes\mu(m))^{-1}(z^p\otimes\mu(m))\\
      &=(z^p[(z^p)^{-1},{m^{t-s-1}}]\otimes\mu(m))^{-1}(z^p\otimes\mu(m))\\
      &=\big(^{z^p}([(z^p)^{-1},{m^{t-s-1}}]\otimes\mu(m))(z^p\otimes\mu(m))\big)^{-1}(z^p\otimes\mu(m)).
   \end{align*}
 Note that ${[(z^p)^{-1},m^{t-s-1}]\in [M^p,M^{p^n}]\leq[M,M]^{p^{n+1}}\leq M^{p^{n+2}}}$ by  Lemma \ref{Shalev identity}. Thus, $[(z^p)^{-1},{m^{t-s-1}}]\otimes\mu(m)\in M^{p^{n+2}}\otimes G$. Therefore, we obtain that ${m^{t-s-1}\otimes [\mu(m),[\mu(m),g]]\equiv 1\mod M^{p^{n+2}}\otimes G}$, and hence the proof.
    \end{proof} 
  
\begin{theorem}\label{T:Tensor of poweful crossed module}
    Let $p$ be an odd prime, $M$ and $G$ be finite $p$-groups, and $\mu:M\to G$ be a crossed module. If $M$ is powerful and $\mu(M)$ is powerfully embedded in $G$, then
    \begin{enumerate}
        \item[(i)] $M^p\otimes G\leqslant (M\otimes G)^p$.
        \item[(ii)] $M\otimes G$ is a powerful finite $p$-group.
    \end{enumerate}
    
\end{theorem}
\begin{proof}
    Let $m\in M$ and $g\in G$. For any $n\geq 1$, setting $r=t=p$ in Proposition \ref{Power_expansion_tensor_identity} yields,
    \begin{equation*}
        m^{p^n}\otimes g=(m^{p^{n-1}})^p\otimes g\equiv (m^{p^{n-1}}\otimes g)^p(m^{p^{n-1}}\otimes[\mu(m^{p^{n-1}}),g])^{\binom{p}{2}}\mod M^{p^{n+1}}\otimes G.
    \end{equation*}
    Since the elements of the form $m^{p^n}\otimes g$ generate $M^{p^n}\otimes G$, using the above inequality, we have ${M^{p^n}\otimes G\leqslant (M\otimes G)^p(M^{p^{n+1}}\otimes G)}$. An inductive argument shows $M^p\otimes G\leqslant (M\otimes G)^p(M^{p^k}\otimes G)$ for any $k\geq 1$. Moreover, choosing $k=\exp(M)$ yields $M^{p^k}\otimes G=1$, and hence proof of \textit{(i)} follows.
    \par By the main result of \cite{E87}, $M\otimes G$ is a finite $p$-group. Using Lemma \ref{commutator_of_ tensor_of_powerful_groups}, we have $\gamma_2(M\otimes G)\leqslant M^p\otimes G.$ Now the proof of $\textit{(ii)}$ follows from $\textit{(i)}.$
\end{proof}
Now we inductively define the $n$-fold tensor product $G^{\otimes n}$. By setting ${M=G^{\otimes 1}=G}$, and choosing $\mu_1$ and $\nu_1$ as the identity on $G$ in \eqref{tenssquare}, we obtain the following crossed square, 
  \begin{equation}
\begin{tikzcd}
		G\otimes G \arrow[r,"\beta_{2}"] \arrow[d,"\alpha_2"] & G \arrow[d,"id"] \\
		G \arrow[r,"\mu_1"] & G
	\end{tikzcd}\label{tensorsquare1}
\end{equation}
where $\alpha_2$ and $\beta_2$ are defined as in \eqref{tenssquare}. Now, set $G^{\otimes 2}=G\otimes G$ and $\mu_2=\beta_2$. By Lemma \ref{BL identity} \textit{(ii)}, the homomorphism $\mu_2:G\otimes G\to G$ is a crossed module.  Assuming $G^{\otimes {n}}$ is defined and $\mu_{n}:G^{\otimes {n}}\to G$ is a crossed module, and taking $\nu_{n}:G\to G$ as identity in \eqref{tenssquare}, we have
      \begin{equation}
\begin{tikzcd}
		G^{\otimes{n}}\otimes G \arrow[r,"\beta_{n+1}"] \arrow[d,"\alpha_{n+1}"] & G \arrow[d,"id"] \\
		G^{\otimes{n}} \arrow[r,"\mu_{n}"] & G
	\end{tikzcd}\label{tensorsquare2}
\end{equation}
For $n\geq 3$, we define $n+1$-fold tensor product $G^{\otimes {n+1}}$= $G^{\otimes{n}}\otimes G$, $\mu_{n+1}=\beta_{n+1}$. It is easy to see that $\mu_{n+1}(x \otimes g)=[\mu_{n}(x),g]$, for all $x \in G^{\otimes{n}}, g \in G$, and $n\geq 1.$

\begin{corollary}\label{nfold tensor is powerful}
    Let $p$ be an odd prime. If $G$ is a finite powerful $p$-group, then for all natural number $n$, the following holds:
    \begin{enumerate}
        \item[(i)] $(G^{\otimes {n}})^p\otimes G\leqslant (G^{\otimes n+1})^p$,
        \item[(ii)] $G^{\otimes {n+1}}$ is finite powerful $p$-group.
    \end{enumerate} 
\end{corollary}
\begin{proof}
      The proof proceeds by induction on $n$. For $n=1$, let $\mu_1:G^{\otimes 1}\to G$ be the identity. Note that the hypotheses of Theorem \ref{T:Tensor of poweful crossed module} are satisfied, and hence $G^p\otimes G\leqslant (G\otimes G)^p$ and $G\otimes G$ is a finite powerful $p$-group. By the induction hypothesis, we have that $G^{\otimes k}$ is a finite powerful $p$-group for $n=k$. Let $\mu_k:G^{\otimes k}\to G$ be defined as above, and note that, $\mu_k(G^{\otimes k})=\gamma_k(G)$ is powerfully embedded in $G$. The proof now follows from Theorem \ref{T:Tensor of poweful crossed module}.
\end{proof}
Now we define the $n$-iterated tensor product $G_{\otimes n}$. Set $G_{\otimes 1}=G$, and inductively define $G_{\otimes n+1}=G_{\otimes n}\otimes G_{\otimes n}, $ for all $n\geq 1.$ Now applying Theorem \ref{T:Tensor of poweful crossed module} inductively to the crossed modules $id:G_{\otimes {n}}\to G_{\otimes {n}}$ we obtain:
\begin{corollary}\label{n-iterated tensor product is powerful}
    Let $p$ be an odd prime. If $G$ is a finite powerful $p$-group, then $G_{\otimes {n+1}}$ is a finite powerful $p$-group for all natural numbers $n$. 
\end{corollary}

\begin{theorem} \label{gamma and derived is finite powerful}
     Let $p$ be an odd prime, suppose we have a group extension $1\to N\to H \to G\to 1$, where $G$ is a finite powerful $p$-group. The following statements hold:
    
    \begin{enumerate}
        \item [(i)] If $N\leqslant Z_n(H)$ for some positive integer $n$, then $\gamma_{n+1}(H)$ is a finite powerful $p$-group.
        \item[(ii)] If $N\leqslant \mathcal{D}_n(H)$ for some positive integer $n$, then $\Gamma_{n+1}(H)$ is a finite powerful $p$-group. 
    \end{enumerate}
\end{theorem}\label{powerful extension}
\begin{proof}
    \textit{(i)} Using Corollary \ref{nfold tensor is powerful}, we have $G^{\otimes n+1}$ is a finite powerful $p$-group. By \cite[Lemma 3.1]{Donadze21}, we have a surjective homomorphism $G^{\otimes n+1}\to \gamma_{n+1}(H)$. Thus, $\gamma_{n+1}(H)$ is a finite powerful $p$-group.\\
    \par \textit{(ii)} Using Corollary \ref{n-iterated tensor product is powerful}, we have $G_{\otimes n+1}$ is a finite powerful $p$-group. By \cite[Lemma 3.7]{Donadze21}, we have a surjective a homomorphism $G_{\otimes n+1}\to \Gamma_{n+1}(H)$. Thus, $\Gamma_{n+1}(H)$ is a finite powerful $p$-group.
\end{proof}

\section{Crossed square of powerful $p$-groups \Romannum{2}.} \label{section: crossed square II}
The non-abelian tensor product modulo $q$ was defined as a generalization of the tensor product in \cite[Definition 1.1]{CR92}. Let $\mu:M \to G$ and $\nu: N \to G$ be two crossed modules and consider the pullback
	\begin{equation}\label{Pull back diagram}
	  \begin{tikzcd}
			M \times_G N \arrow[r,"\pi_2"] \arrow[d,"\pi_1"] & N \arrow[d,"\nu"] \\
			M \arrow[r,"\mu"] & G
		\end{tikzcd}
	\end{equation}
	where $K=M \times_G N=\{(m,n) \mid m \in M,n\in N,\mu(m)=\nu(n)\}$. In this diagram, each group acts on any other group via its image in the group $G$. 
	\begin{definition}\label{Defn q tensor}
		The tensor product modulo q, $M \tens N$, of the crossed modules $\mu$ and $\nu$ is the group generated by the symbols $m \otimes n$ and $\{k\}$, $m \in M$, $n \in N$, $k \in K$, with relations
		\begin{align}
			& m \otimes nn'=(m \otimes n)(\prescript{n}{}{m} \otimes \prescript{n}{}{n'})\label{E:q-tensor def 1} \\
			& mm' \otimes n = (\prescript{m}{}{m'} \otimes \prescript{m}{}{n})(m \otimes n)\\
			& \{k\}(m \otimes n)\{k\}^{-1}= \prescript{k^q}{}{m} \otimes \prescript{k^q}{}{n} \label{E:q-tensor def 3}\\
			&\{kk'\}=\{k\} \prod \limits_{i=1}^{q-1}(\pi_1k^{-1} \otimes ( \prescript{k^{1-q+i}}{}{\pi_2k'})^i)\{k'\}  \\
			& [\{k\},\{k'\}]=\pi_1 k^q \otimes \pi_2 {k'}^q \label{E:q-tensor def 5} \\
			& \{(m\prescript{n}{}{m^{-1}},\prescript{m}{}{n}n^{-1})\}=(m \otimes n)^q\label{E:q-tensor def 6} 
		\end{align}
   for all $m,m' \in M$, $n,n' \in N$, $k, k' \in K$.
	\end{definition}
We have two group homomorphisms $\alpha:M \tens N \to M$ and $\beta: M \tens N \to N$ defined by 
\begin{align}
	& \alpha(m \otimes n)=m\prescript{n}{}{m^{-1}}, & \alpha(\{k\})=\pi_1 k^q,\label{E:alpha}\\
	& \beta(m \otimes n)=\prescript{m}{}{n}n^{-1}, & \beta(\{k\})=\pi_2 k^q.\label{E:beta}
\end{align}
We also have actions of the groups $G$ on the group $M \tens N$ given by $\prescript{g}{}{(m \otimes n)}=(\prescript{g}{}{m} \otimes \prescript{g}{}{n})$ and $\prescript{g}{}{\{k\}}=\{\prescript{g}{}{k}\}$ for $m \in M$, $n \in N$, $k \in K$ and $g \in G$. These actions commute with relations \eqref{E:q-tensor def 1} to \eqref{E:q-tensor def 6}. 
\begin{prop} [Proposition 1.16, \cite{CR92}]\label{qtensor is crossed square}
     Let $\mu:M\to G$ and $\nu:N\to G$ be crossed modules, which induce the actions  of $M$ and $N$ on themselves and on each other via $G$. Then there is a crossed square 
\begin{equation}
    \begin{tikzcd} 
	M \tens N \arrow[r,"\beta"]\arrow[d,"\alpha"] & N\arrow[d,"\nu"] 
	\\ M \arrow[r, "\mu"]& G
    \end{tikzcd}\label{qtensorsquare}
\end{equation}
where $\alpha(m\otimes n)=m \prescript{n}{}{m^{-1}}$, $\alpha(\{k\})= \pi_1k^q$, $\beta(m\otimes n)= \prescript{m}{}{n}n^{-1}$, ${\beta(\{k\})=\pi_2k^q}$ and $h(m,n)=m\otimes n$.    
\end{prop}

 Setting $q=p$, $N=G$, and $\nu:N\to G $ as identity in \eqref{qtensorsquare}, any crossed module $\mu : M \to G$ yields the following crossed square 
\begin{equation}
	\begin{tikzcd} \label{Gqtensorcrossedsquare}
		M \otimes^p G \arrow[r,"\beta"]\arrow[d,"\alpha"] & G\arrow[d,"id"] 
		\\ M \arrow[r, "\mu"]& G
	\end{tikzcd}
\end{equation}where $\alpha,\beta$ are defined as in \eqref{E:alpha}, \eqref{E:beta}, respectively. Composing the natural homomorphism $\tau_1 : M^p \otimes G \to M\otimes G$ with $\sigma: M \otimes G \to M \ptens G$ defined in \cite[Proposition 1.6]{CR92}, we obtain a homomorphism $\eta: M^p \otimes G \to M \otimes^p G$. 
\begin{equation}
    \begin{tikzcd}
        M^{p}\otimes G\arrow[r, "\tau_1"]\arrow[dr,swap, "\eta"]&M\otimes G\arrow[d,"\sigma"]\\
        &M\otimes^p G
    \end{tikzcd}
\end{equation}We have that $\eta (M^p \otimes G)$ is a $G$-invariant subgroup, and hence a normal subgroup of $M \ptens G$.

\begin{lemma} \label{commutatorMqtensorG}
	Let $p$ be an odd prime, $M$ and $G$ be finite $p$-groups, and $\mu:M\to G$ be a crossed module. If $M$ is powerful and $\mu(M)$ is powerfully embedded in $G$, then $\gamma_2(M \ptens G) \leqslant \eta(M^p \otimes G)$.
\end{lemma}

\begin{proof}
Since $\eta(M^p \otimes G)$ is a $G$-invariant subgroup, it is a normal subgroup of $M \ptens G$, and hence, it is enough to show that $[x,y] \in \eta(M^p \otimes G)$, where $x$ and $y$ are generators of $M \ptens G$. We begin the proof by showing that $[m \otimes g,m_1 \otimes g_1] \in \eta(M^p \otimes G)$ for all $m,m_1 \in M$ and $g,g_1 \in G$. We have that ${[m \otimes g,m_1 \otimes g_1]=\sigma([m \otimes g,m_1 \otimes g_1])}$. Using Lemma \ref{commutator_of_ tensor_of_powerful_groups}, we have that $\sigma([m \otimes g,m_1 \otimes g_1]) \in \sigma(\tau_1(M^p \otimes G))$. This gives us that ${\sigma([m \otimes g,m_1 \otimes g_1]) \in \eta(M^p \otimes G)}$. By \eqref{E:q-tensor def 5}, for all $\{k\},\{k'\} \in M \ptens G$, we have ${[\{k\},\{k'\}] = \pi_1(k)^p \otimes \pi_2(k')^p}$. Note that $[\{k\},\{k'\}] \in \eta(M^p \otimes G)$, as $\pi_1(k)^p \otimes \pi_2(k')^p=\eta\big(\pi_1(k)^p \otimes \pi_2(k')^p\big)$.
To complete the proof, we have to show that $[\{k\},m \otimes g] \in \eta(M^p \otimes G)$ where $k=(m_1,g_1) \in K$, $m,m_1 \in M$ and $g,g_1 \in G$. Recall that $K$ acts on $M$ and $G$ via \eqref{Pull back diagram}, i.e. $\prescript{k}{}{m}=\prescript{\mu(m_1)}{}{m}$ and $\prescript{k}{}{g}= \prescript{g_1}{}{g}$. By \eqref{E:q-tensor def 3}, \eqref{E:Crossed module 2} and the relation $aba^{-1}=b[b^{-1},a]$, we obtain that 
	\begin{align*}
		[\{k\},m \otimes g] =& ( \prescript{k^p}{}{m} \otimes \prescript{k^p}{}{g})(m \otimes g)^{-1} \\
  =& ( m_1^pm(m_1^p)^{-1} \otimes g_1^{p}g{g_1^p}^{-1})(m \otimes g)^{-1} \\
		=&( m[m^{-1},m_1^p] \otimes g[g^{-1},g_1^p])(m \otimes g)^{-1} \\
		=& \prescript{m}{}{( [m^{-1},m_1^p] \otimes g[g^{-1},g_1^p])}( m \otimes \ g[g^{-1},g_1^p])(m \otimes g)^{-1}.
	\end{align*}
	Since $M$ is powerful, we have that $[m^{-1},m_1^p] \in \gamma_2(M) \leqslant M^p$. Therefore ${\prescript{m}{}{( [m^{-1},m_1^p] \otimes \ g[g^{-1},g_1^p])} \in \eta(M^p \otimes G)}$ because
    $M^p$ is a characteristic subgroup of $M$.
    Thus we have that
	\begin{align*}
		[\{k\},m \otimes g]&\equiv ( m \otimes \ g[g^{-1},g_1^p])(m \otimes g)^{-1} \mod \eta(M^p\otimes G) \\
		& \equiv (m \otimes g) \prescript{g}{}{( m \otimes \ [g^{-1},g_1^p])}(m \otimes g)^{-1} \mod \eta(M^p\otimes G).
	\end{align*}
	Since $M$ is $G$-invariant, it is enough to show that $m \otimes \ [g^{-1},g_1^p] \in \eta(M^p \otimes G)$. Note that $(m_1,g_1) \in K$ and hence $g_1=\mu(m_1)$ by the definition of $K$. Therefore
	\begin{align*}
		m \otimes \ [g^{-1},g_1^p]&=m \otimes \ [g^{-1},\mu(m_1^p)] \\
		&= m \otimes \ [\mu(m_1^p),g^{-1}]^{-1} \\
		&= m \otimes \ (\prescript{\mu(m_1^p)}{}{g^{-1}}g)^{-1} \\
		&= m \otimes \ (\prescript{m_1^p}{}{g^{-1}}g)^{-1} \\
		&= m \otimes \ \beta(m_1^p \otimes g^{-1})^{-1}  && \text{(Lemma } \ref{BL identity}) \\
		&=\prescript{m}{}{(m_1^p \otimes g^{-1})^{-1}}(m_1^p \otimes g^{-1}).
	\end{align*}
	Since $m_1^p \otimes g^{-1} \in \eta(M^p \otimes G)$, we have that $m \otimes \ [g^{-1},g_1^p] \in \eta(M^p \otimes G)$ and hence the proof.
\end{proof}


\begin{theorem} \label{qtensorpowerful}
 Let $p$ be an odd prime, $M$ and $G$ be finite $p$-groups, and $\mu:M\to G$ be a crossed module. If $M$ is powerful and $\mu(M)$ is powerfully embedded in $G$, then $M \otimes^p G$ is a finite powerful $p$-group.
\end{theorem}
\begin{proof}
 By \cite[Corollary 12]{MR1088877}, $M \ptens G$ is a finite group. Using the main results of \cite{E87} and \cite[Propositon 11]{MR1088877}, we obtain that $M \ptens G$ is a $p$-group.
	By Lemma \ref{commutatorMqtensorG}, we have that $\gamma_2(M \otimes^p G) \leqslant \eta(M^p \otimes G)$. Using Theorem~\ref{T:Tensor of poweful crossed module}\textit{(i)}, we have that $\tau_1(M^p \otimes G)\leqslant (M \otimes G)^p$ and therefore $ {\eta(M^p \otimes G) \leqslant \sigma((M \otimes G)^p)}$. By Theorem~\ref{T:Tensor of poweful crossed module}, $M \otimes G$ is powerful and hence $(M \otimes G)^p=\{x^p \mid x \in M \otimes G\}$. Since $\sigma(\{x^p \mid x \in M \otimes G)\} \leqslant (M \ptens G)^p$, we have that ${\sigma((M \otimes G)^p) \leqslant (M \ptens G)^p}$. Thus $M \ptens G$ is powerful. 
\end{proof}
Now we inductively define the $n$-fold $q$-tensor product $G^{\otimes^q_{n}}$. 
By setting $M=N=G^{\otimes^q_1}=G$, and choosing $\mu_1$ and $\nu_1$ as the identity map on $G$ in \eqref{qtensorsquare}, we obtain the following crossed square, 
\begin{equation}
\begin{tikzcd}
		G\tens G \arrow[r,"\beta_{2}"] \arrow[d,"\alpha_2"] & G \arrow[d,"id"] \\
		G \arrow[r,"\mu_1"] & G
	\end{tikzcd}\label{n fold qtensorsquare1}
\end{equation}
where $\alpha_2,\beta_2$ are defined as in \eqref{E:alpha}, \eqref{E:beta}, respectively. Set $G^{\otimes^q_2}=G\tens G$, $\mu_2=\beta_2$.
By Proposition \ref{qtensor is crossed square}, the homomorphism $\mu_2: G \tens G  \to G$ is a crossed module.
Assuming $G^{\otimes^q_{n-1}}$ is defined, and $\mu_{n-1}:G^{\otimes^q_{n-1}}\to G$ is a crossed module, and taking $\nu_{n-1}:G\to G$ as identity in \eqref{qtensorsquare}, we have
      \begin{equation}
\begin{tikzcd}
		G^{\otimes^q_{n-1}}\tens G \arrow[r,"\beta_{n}"] \arrow[d,"\alpha_n"] & G \arrow[d,"id"] \\
		G^{\otimes^q_{n-1}} \arrow[r,"\mu_{n-1}"] & G
	\end{tikzcd}\label{q n fold tensorsquare}
\end{equation}
For $n\geq 3$, we now define $n$-fold $q$-tensor product $G^{\otimes^q_{n}}= G^{\otimes^q_{n-1}}\tens G$, $\mu_n=\beta_n$. It is easy to see that for all $x \in G^{\otimes^q_{n-1}}$ and $g \in G$,
\begin{itemize}
    \item[\textit{(i)}] $\mu_n(x \otimes g)=[\mu_{n-1}(x),g]$
    \item[\textit{(ii)}] $\mu_n(\{(x,\mu_{n-1}(x)),(x,\mu_{n-1}(x))\})={\mu_{n-1}(x)}^q$.
\end{itemize}

\begin{corollary} \label{Gqtensorpowerful}
	Let $p$ be an odd prime. Suppose $G$ is a finite powerful $p$-group, then $G^{\otimes^p_{n+1}}$ is a finite powerful $p$-group for all $n\geq 1$.
\end{corollary}

\begin{proof}
    The proof proceeds by induction on $n$. For $n=1$, let $\mu_1: G \to G$ be the identity map. Note that the hypotheses of Theorem \ref{qtensorpowerful} are satisfied, and hence $G \ptens G$ is a finite powerful $p$-group. Assume for $n=k$, $G^{\ptens_k}$ is a finite powerful $p$-group. Let $\mu_k$ be the crossed module defined as above. Note that $\mu_k(G^{\ptens_k})=P_k(G)$, where $P_k(G)$ is the $k$-th term of the lower $p$-central series of $G$. By Theorem \ref{MannLubotzky_powerfully_embedded}\textit{(ii)}, $P_k(G)$ is powerfully embedded in $G$. Now the proof follows by Theorem \ref{qtensorpowerful}.    
\end{proof}

Now we define $n$-iterated $q$-tensor product $G_{\otimes^q_{n}}$. Set $G_{\otimes^q _{1}}=G$, and inductively define 

\begin{equation}\label{iteratedqtensordefn}
    G_{\otimes^q_{n+1}}=G_{\otimes^q_{n}}\otimes^q G_{\otimes^q_{n}},
\end{equation} for all $n\geq 1.$ Applying Theorem \ref{qtensorpowerful} inductively to the crossed module $id:G_{\otimes^p_{n}}\to G_{\otimes^p_{n}},$ we have
\begin{corollary}\label{n-iterated qtensor product is powerful}
    Let $p$ be an odd prime. If $G$ is a finite powerful $p$-group, then $G_{\otimes^p_{n+1}}$ is finite powerful $p$-group for $n\geq 1$. 
\end{corollary}

\section{Morphisms between crossed squares} \label{section: crossed square morphisms}
In the next proposition, we will define a group homomorphism between two crossed modules that arise from two crossed squares.

\begin{prop}\label{The commutative cube}
    Let $\mu:A \to H$ and $\gamma:C \to H$ be $H$-crossed modules, $\nu:B \to G$ and $\delta:D \to G$ be $G$-crossed modules. Assume $\phi: H \to G$, $f_1: A \to B$, and $f_2: C \to D$ are group homomorphisms satisfying the following compatibility conditions
    \begin{equation}\label{compatability conditions}
        f_1(\prescript{h}{}{a})=\prescript{\phi(h)}{}{f_1(a)} \text{ and } 
        f_2(\prescript{h}{}{c})=\prescript{\phi(h)}{}{f_2(c)},
    \end{equation}
    for all $h\in H$, $a\in A$, $c\in C$. If the diagrams
    \begin{center}
			\begin{tikzcd} 
			A \arrow[r,"f_1"]\arrow[d,"\mu"] & B \arrow[d,"\nu"]
			\\\ H \arrow[r, "\phi"]& G
			\end{tikzcd}
   \qquad and \qquad\begin{tikzcd}
			C \arrow[r,"f_2"]\arrow[d,"\gamma"] & D \arrow[d,"\delta"]
			\\\ H \arrow[r, "\phi"]& G
   \end{tikzcd}
    \end{center}
   are commutative, then the following statements hold: 
 \begin{enumerate}
     \item [(i)] There exists a group homomorphism $f_1 \tens f_2: A \tens C \to B \tens D$ defined by 
     \begin{align*}
         (f_1 \tens f_2)(a \otimes c)&=f_1(a) \otimes f_2(c) \\
         (f_1 \tens f_2) \big(\big\{(a_1,c_1) \big\}\big)&=\big\{\big(f_1(a_1),f_2(c_1)\big)\big\}
     \end{align*}
     for all $a\in A$, $c \in C$ and $\{(a_1,c_1)\} \in K$ (\textit{cf.} Definition \ref{Defn q tensor}).
     \item[(ii)] The group homomorphism $f_1 \tens f_2$ satisfies the compatibility condition, i.e.,
     \[(f_1 \tens f_2)(\prescript{h}{}{x}) =\prescript{\phi(h)}{}{(f_1 \tens f_2)(x)}\]
     for all $x \in A \tens B$ and $h \in H$.
     \item[(iii)] The following cube is commutative
\begin{center}
    \begin{tikzcd}[row sep=1.5em, column sep = 1.5em]
    A \tens C \arrow[rr,"f_1 \tens f_2"] \arrow[dr, swap,"\beta_1"] \arrow[dd, swap,"\alpha_1"] &&
    B \tens D \arrow[dd,"\alpha_2", pos=.3] \arrow[dr,"\beta_2"] \\
    & C \arrow[rr, "f_2", pos=.3, crossing over] &&
    D \arrow[dd,"\delta"] \\
    A \arrow[rr,"f_1",pos=.3] \arrow[dr, "\mu"] && B \arrow[dr,"\nu"] \\
    & H \arrow[rr,"\phi"] \arrow[from=uu,"\gamma", pos=.3, crossing over]&& G
    \end{tikzcd}
\end{center}
where $\alpha_1,\beta_1$ and $\alpha_2,\beta_2$ are obtained from the corresponding crossed squares (\textit{cf.} Proposition \ref{qtensor is crossed square}).
\item [(iv)]In particular, the following diagram is commutative:
\begin{center}
			\begin{tikzcd} 
			A \tens C \arrow[r,"f_1 \tens f_2"]\arrow[d,"\gamma \circ \beta_1"] & B \tens D \arrow[d,"\delta \circ \beta_2"]
			\\\ H \arrow[r, "\phi" ]& G
			\end{tikzcd}
    \end{center}
 \end{enumerate}

\end{prop}
 \begin{proof}
 
        \textit{(i)} Let $\{(a_1,c_1)\} \in K=A \times_H C$, hence we have that ${\mu(a_1)=\gamma_(c_1)}$. Note that $\nu(f_1(a_1))=\phi(\mu(a_1))=\phi(\gamma(c_1))=\delta(f_2(c_1))$. Therefore, we obtain that ${(f_1 \tens f_2) (\{(a_1,c_1)\})=\{(f_1(a_1),f_2(c_1))\}}$ is an element of $B \tens D$. To see that $f_1 \tens f_2$ is a well-defined group homomorphism, we need to see that $f_1 \tens f_2$ preserves the relations \eqref{E:q-tensor def 1}-\eqref{E:q-tensor def 6}. We will show that the relation \eqref{E:q-tensor def 6} is preserved by $f_1 \tens f_2$, and the other relations can be shown similarly. Let $a \in A$ and $c \in C$. Recalling that $A$ acts on $C$ via $\mu$, $C$ acts on $A$ via $\gamma$, $B$ acts on $D$ via $\nu$ and $D$ acts on $B$ via $\delta$, we obtain 
        \begin{align*}
            (f_1 \tens f_2)(\{(a &\prescript{c}{}{a^{-1}},\prescript{a}{}{c}c^{-1})\}) \\
            &= \{(f_1(a)f_1(\prescript{c}{}{a^{-1}}),f_2(\prescript{a}{}{c})f_2(c^{-1}))\} \\
            &=\{(f_1(a)f_1(\prescript{\gamma(c)}{}{a^{-1}}),f_2(\prescript{\mu(a)}{}{c})f_2(c^{-1}))\} \\
            &=\{(f_1(a)\prescript{\phi(\gamma(c))}{}{f_1(a)^{-1}},\prescript{\phi(\mu(a))}{}{f_2(c)}f_2(c)^{-1})\} &\mbox{(By } \eqref{compatability conditions})\\
          &=\{(f_1(a)\prescript{\delta(f_2(c))}{}{f_1(a)^{-1}},\prescript{\nu(f_1(a))}{}{f_2(c)}f_2(c)^{-1})\} \\
            &=\{(f_1(a)\prescript{f_2(c)}{}{f_1(a)^{-1}},\prescript{f_1(a)}{}{f_2(c)}f_2(c)^{-1})\} \\
            &=(f_1(a) \otimes f_2(c))^q & \mbox{(By } \eqref{E:q-tensor def 6})\\
            & = (f_1 \ptens f_2)((a \otimes c)^q).
        \end{align*}
        \par\textit{(ii)} The proof follows easily using the compatibility conditions of $f_1$ and $f_2$ given in \eqref{compatability conditions}.
        \par\textit{(iii)} Using Proposition \ref{qtensor is crossed square}, it is enough to show that the diagram
        \begin{center}
			\begin{tikzcd} 
			A \tens C \arrow[r,"f_1 \tens f_2"]\arrow[d," \beta_1"] & B \tens D \arrow[d,"\beta_2"]
			\\
   \ C \arrow[r, "f_2" ]& D
			\end{tikzcd}
   \quad and \quad\begin{tikzcd} 
			A \tens C \arrow[r,"f_1 \tens f_2"]\arrow[d,"\alpha_1"] & B \tens D \arrow[d,"\alpha_2"]
			\\\ A \arrow[r, "f_1" ]& B
			\end{tikzcd}
    \end{center}
    are commutative. We will show the commutativity of the left square, and the commutativity of the right square follows similarly. Towards that end, we have that $f_2(\beta_1(a \otimes c))=f_2(\prescript{a}{}{c}c^{-1})=f_2(\prescript{a}{}{c})f_2(c)^{-1}$. Since $A$ acts on $C$ via $\mu$, we have that $f_2(\prescript{a}{}{c})=f_2(\prescript{\mu(a)}{}{c})$. By \eqref{compatability conditions}, we have that ${f_2(\prescript{\mu(a)}{}{c})=\prescript{\phi(\mu(a))}{}{f_2(c)}}$. Therefore, we obtain $f_2(\beta_1(a \otimes c))=\prescript{\phi(\mu(a))}{}{f_2(c)}f_2(c)^{-1}$. On the other hand we have $\beta_2 \circ (f_1 \tens f_2)(a \otimes c)=\prescript{f_1(a)}{}{f_2(c)}f_2(c)^{-1}$. Since $B$ acts on $D$ via $\nu$ and $\nu \circ f_1=\phi \circ \mu$, we have that 
    \begin{equation*}
        \prescript{f_1(a)}{}{f_2(c)}f_2(c)^{-1}=\prescript{\nu(f_1(a))}{}{f_2(c)}f_2(c)^{-1}=\prescript{\phi(\mu(a))}{}{f_2(c)}f_2(c)^{-1}.
    \end{equation*}
    Thus $f_2 \circ \beta_1 (a \otimes c)=\beta_2 \circ (f_1 \tens f_2)(a \otimes c)$. Moreover, we have that $f_2(\beta_1(\{(a_1,c_1)\}))=f_2(c_1^q)=f_2(c_1)^q$. Note that 
    \begin{equation*}
        \beta_2 (f_1 \tens f_2(\{(a_1,c_1)\}))=\beta_2(\{(f_1(a_1),f_2(c_1))\})=f_2(c_1)^q.
    \end{equation*}
    Thus $f_2 \circ \beta_1 = \beta_2 \circ (f_1 \tens f_2)$, and hence the proof.
    \par\textit{(iv)} This follows from \textit{(iii)}.
 \end{proof}

 \begin{lemma}\label{Lemma_Hinvariant}
    Let $\mu:M\to H$ and $\nu:N\to H$ be crossed modules, and let $R$ and $S$ be $H$-invariant subgroups of $M$ and $N$, respectively. If $i:R\to M$ and $j:S\to N$ be inclusion maps, then 
    \begin{enumerate}
        \item [(i)]  $(i\tens j)(R\tens S)$ is an $H$-invariant subgroup of $M\tens N$.
        \item [(ii)] $(i\tens j)(R\tens S)$ is a normal subgroup of $M\tens N$.
    \end{enumerate}
    
\end{lemma}
\begin{proof} 
\textit{(i)} Let $r\otimes s, \{(r_1,s_1)\}\in \im{i\tens j},$ and $h\in H$. We have $\prescript{h}{}{(r\otimes s)}=\prescript{h}{}{r}\otimes \prescript{h}{}{s}$. Since $R$, $S$ are $H$-invariant subgroups, $\prescript{h}{}{r}\in R$ and $\prescript{h}{}{s}\in S$. Thus $\prescript{h}{}{(r\otimes s)}\in \im{i\tens j}$. Similarly $\prescript{h}{}{\{(r_1,s_1)\}}\in \im{i\tens j}$, and hence the proof.

\par\textit{(ii)} The proof follows from Proposition \ref{qtensor is crossed square} and Lemma \ref{Ginvariant is normal}.
\end{proof}
 Let $N$ be a normal subgroup of a group $H$. Set $\hat{N}_1^1=N$ and define $\mu_1^1:\hat{N}_1^1=N\to H$ as the inclusion map. Similar to the $n$-fold $q$-tensor product constructed in Section 4, we can construct a crossed module  
 \[\mu_n^i:H \tens\ldots \tens H \tens N \tens H \tens\ldots \tens H \to H,\]
 where $N$ appears in the $i$-th position for $1\leq i\leq n$, while the remaining $n-1$ terms are $H$. We denote $H \tens\ldots \tens H \tens N \tens H \tens \ldots \tens H$ by $\hat{N}^{i}_n$.
 \begin{lemma}\label{{N}_i normality}

    Let $N$ be a normal subgroup of a group $G$ and $n$ be a positive integer. For $1\leq i\leq n,$ there is an $H$-module homomorphism ${f_i^n:\hat{N}_n^{i}\to H^{\otimes^q_{n}}}$ such that $\im{f_i^n}$ is an $H$-invariant subgroup and the following diagram commutes \begin{equation}
      	\begin{tikzcd} 	
			\hat{N}^i_{n} \arrow[r,"f_i^n"]\arrow[d,"\mu ^{i}_{n}"] & H^{\otimes^{q}_{n}} \arrow[d,"\mu_n"] 
			\\\ H \arrow[r,"id"]& H
			\end{tikzcd} \label{N_i commutativity Lemma}
  \end{equation}    
\end{lemma}
 \begin{proof}
     The proof proceeds by induction on $n$. For $n=1$, set $f_1^1:\hat{N}^1_{1} \to H^{\otimes^{q}_{1}}$ as the inclusion map. Clearly $f_1^1$ is an $H$-module morphism, $\im{f_1^1}$ is an $H$-invariant subgroup and the diagram \eqref{N_i commutativity Lemma} commutes. For $n>1$ and $i<n$, setting $A=\hat{N}^i_{n-1}$, $\mu=\mu_{n-1}^i$, $B=H^{\otimes^q_{n-1}}$, $\nu=\mu_{n-1}$, $C=D=G=H$, $\gamma=\delta=id$, $f_1=f_i^{n-1}$, $f_2=id$ in Proposition \ref{The commutative cube}, we have the following commutative cube:

 \begin{equation}\label{special cube}
   \begin{tikzcd}[row sep=1.5em, column sep = 1.5em]
  \hat{N}^i_{n-1}\tens H \arrow[rr,"f_i^{n-1}\tens id"] \arrow[dr,swap,"\mu_n^i"] \arrow[dd, swap] &&
  H^{\otimes^q_{n-1}}\tens H \arrow[dd, pos=.3] \arrow[dr,"\mu_n"] \\
   & H \arrow[rr, pos=.3,"id", crossing over] &&
   H \arrow[dd,"id"] \\
\hat{N}^i_{n-1} \arrow[rr,"f_i^{n-1}",pos=.3] \arrow[dr, swap,"\mu_{n-1}^i"] && H^{\otimes^q_{n-1}}\arrow[dr,"\mu_{n-1}"] \\
   & H \arrow[rr,"id"] \arrow[from=uu,"id", pos=.3, crossing over]&& H
    \end{tikzcd}
\end{equation}Define $f_i^n=f_i^{n-1}\tens id$, and note that by Proposition \ref{The commutative cube}\textit{(ii)}, $f_i^n$ is an $H$-module homomorphism. The commutativity of $\eqref{N_i commutativity Lemma}$ follows from \eqref{special cube} and Proposition \ref{The commutative cube}\textit{(iv)}. Using \eqref{N_i commutativity Lemma}, it can be shown that \[f_i^{n-1}\tens id:\hat{N}^i_{n-1}\tens H \to f_i^{n-1}(\hat{N}^i_{n-1})\tens H\] is surjective. Moreover, we have the following commutative diagram,

\[
\begin{tikzcd}
 & f_i^{n-1}(\hat{N}^i_{n-1})\tens H \arrow{dr}{j\tens id} \\
\hat{N}^i_{n-1}\tens H \arrow{ur}{f_i^{n-1}\tens id} \arrow{rr}{f_i^{n}} && H^{\otimes^q_{n-1}}\tens H
\end{tikzcd}
\]where $j:f_i^{n-1}(\hat{N}^i_{n-1})\to H^{\otimes^q_{n-1}}$ is the inclusion map. Therefore, we have $\im{j\otimes id}=\im{f_i^n}$. Hence, by Lemma \ref{Lemma_Hinvariant}\textit{(i)}, $\im{f_i^n}$ is an $H$-invariant subgroup. For $n>1$ and $i=n$, setting $A=B=H^{\otimes^q_{n-1}}$, $\mu=\nu=\mu_{n-1}$, $D=G=H$, $C=N$, $f_1=id$, $\phi=\delta=id$ and $f_2=\gamma=j$ in Proposition \ref{The commutative cube}, where $j$ is the inclusion map,  we obtain the following commutative diagram: 
 \begin{equation}\label{special cube1}
   \begin{tikzcd}[row sep=1.5em, column sep = 1.5em]
   H^{\otimes^q_{n-1}}\tens N\arrow[rr,"id\;\tens \;j"] \arrow[dr,swap,"\mu_n^n"] \arrow[dd, swap] &&
  H^{\otimes^q_{n-1}}\tens H \arrow[dd, pos=.3] \arrow[dr,"\mu_n"] \\
   & N \arrow[rr, pos=.3,"j", crossing over] &&
   H \arrow[dd,"id"] \\
 H^{\otimes^q_{n-1}} \arrow[rr,"id",pos=.3] \arrow[dr, swap,"\mu_{n-1}"] && H^{\otimes^q_{n-1}}\arrow[dr,"\mu_{n-1}"] \\        
   & H \arrow[rr,"id"] \arrow[from=uu,"j", pos=.3, crossing over]&& H
    \end{tikzcd}
 \end{equation}Define $f_i^n=id\tens j$, and note that by Proposition \ref{The commutative cube}\textit{(ii)}, $f_i^n$ is an $H$-module homomorphism. The commutativity of $\eqref{N_i commutativity Lemma}$ follows from \eqref{special cube1} and Proposition \ref{The commutative cube}\textit{(iv)}. By Lemma \ref{Lemma_Hinvariant}\textit{(i)}, $\im{id\tens j}$ is an $H$-invariant subgroup and hence the proof.
     
 \end{proof}

For a normal subgroup $N$ of $H$, we will now inductively define a homomorphism ${\phi_n: H^{\tens_n} \longrightarrow(H/N)^{\tens_n}}$ satisfying the compatibility condition $\phi_n(\prescript{h}{}{x})=\prescript{\phi_1(h)}{}{\phi_n(x)}$. For $n=1$, $\phi_1$ is the natural projection map which satisfies the compatibility condition. Assume that $\phi_{n-1}: H^{\tens_{n-1}} \longrightarrow (H/N)^{\tens_{n-1}}$ is well defined and satisfies the compatibility condition. By \eqref{q n fold tensorsquare}, we have the following crossed modules $\mu_{n-1}^{H}: H^{\tens_{n-1}} \longrightarrow H$, $\mu_{n-1}^{H/N}: (H/N)^{\tens_{n-1}} \longrightarrow H/N$. Setting $A=H^{\tens_{n-1}}$, $\mu= \mu_{n-1}^H$, $B=(H/N)^{\tens_{n-1}}$, $\nu=\mu_{n-1}^{H/N}$, $C=H$, $\gamma=id_H$, $D=G=H/N$, $\delta=id_{H/N}$, $f_1=\phi_{n-1}$, $f_2=\phi=\phi_1$ in Proposition \ref{The commutative cube}, we obtain the map $\phi_{n-1} \tens \phi_1$ which satisfies the compatibility condition. We denote this map by $\phi_n$. Moreover, the maps $\beta_1$ and $\beta_2$ from Proposition \ref{The commutative cube} are the maps $\mu_n^H$ and $\mu_n^{H/N}$ defined as in \eqref{q n fold tensorsquare}.

\begin{equation} \label{special cube for exact sequence for Lazard}
    \begin{tikzcd}[row sep=1.5em, column sep = 1.5em]
    H^{\tens_{n-1}} \tens H \arrow[rr,"\phi_{n-1} \tens \phi_1"] \arrow[dr, swap,"\beta_1"] \arrow[dd, swap,"\alpha_1"] &&
    (H/N)^{\tens_{n-1}} \tens H/N \arrow[dd,"\alpha_2", pos=.3] \arrow[dr,"\beta_2"] \\
    & H \arrow[rr, "\phi_1", pos=.3, crossing over] &&
    H/N \arrow[dd,"id_{H/N}"] \\
    H^{\tens_{n-1}} \arrow[rr,"\phi_{n-1}",pos=.3] \arrow[dr, "\mu_{n-1}^H"] && (H/N)^{\tens_{n-1}} \arrow[dr,"\mu_{n-1}^{H/N}"] \\
    & H \arrow[rr,"\phi_1"] \arrow[from=uu,"id_H", pos=.3, crossing over]&& H/N
    \end{tikzcd}
\end{equation}

By Lemma \ref{{N}_i normality}, there is an $H$-module homomorphism $\hat{f_i^n}:\hat{N}_n^i \to H^{\tens_n}$. The next Proposition extends a result by Ellis \cite[Proposition 7]{Ellis1995}. Setting $N_i= Im (\hat{f_i^n})$, we have

 \begin{prop} \label{nfold_qtensor_exact_sequence}
     Let $N$ be a normal subgroup of $H$ and $n$ be a positive integer. Set $G:=H/N$. The following sequence is exact: 
        \begin{equation*}
		   1 \to \prod \limits_{i=1}^n N_i \to H^{\otimes^{q}_n} \to G^{\otimes^{q}_n} \to 1.
        \end{equation*}
 \end{prop}

 \begin{proof}
     Using Proposition \ref{The commutative cube} and some effort, the proof follows.
 \end{proof}

\section{Finiteness of the terms of the lower $p$-series for a pro-$p$ group} \label{section: finiteness of lower p series}
 For an arbitrary group $G$, set $\lambda_n(G)=\gamma_1(G)^{p^{n-1}}\gamma_2(G)^{p^{n-2}} \ldots \gamma_n(G)$ for all $n\geq 1$. It is well known that $\lambda_n(G)$ is a characteristic subgroup of $G$ and $G=\lambda_1(G)$ and $\lambda_{n+1}(G)\leqslant\lambda_n(G)$. It can also be defined as $\lambda_1(G)= G$ and $\lambda_n(G)=\lambda_{n-1}(G)^p[\lambda_{n-1}(G),G] $ for $n> 1$
 . This series is called the lower $p$-series or the Lazard series. By Definition 1.15 in \cite{DDMS99}, for a pro-$p$ group G, the lower $p$-series is defined as $P_1(G)=G$ and $P_{n}(G)=\overline{P_{n-1}(G)^p} \overline{[P_{n-1}(G),G]} $ for all $n>1$.


The following easy lemma gives the relation between the two definitions.
\begin{lemma}\label{PR:L:1}
        Let $G$ be a pro-$p$ group, we have $P_n(G)= \overline{\lambda_n(G)}.$
\end{lemma}
\begin{proof}
        This follows by induction on $n$ and Lemma \ref{PR:R:2}.
\end{proof}

 \begin{lemma}\label{MR:L:1}
Let $H$ be a pro-$p$ group. If $N$ is a normal subgroup of $Z_n(H)$ with exponent $p$, then 
\begin{enumerate}
    \item[(i)] $[\lambda_i(H), N,\prescript{}{n-i}{H}]=1$ for $1\leq i\leq n$.
    \item[(ii)] $[N,\prescript{}{n}{H}]=1$.
\end{enumerate}
 \end{lemma}
\begin{proof}\textit{(i)} Since $\gamma_i(H)^{p^k}$ is normal subgroup for any non negative integer $k$, Lemma \ref{PR:R:1}\textit{(i)} yields \begin{align*}
    [\lambda_i(H), N,\prescript{}{n-i}{H}]&= [\gamma_1(H)^{p^{i-1}}\gamma_2(H)^{p^{i-1}}\ldots \gamma_i(H), N,\prescript{}{n-i}{H}]\\
    &= \big[[\gamma_1(H)^{p^{i-1}},N][\gamma_2(H)^{p^{i-1}},N]\ldots [\gamma_i(H), N], \prescript{}{n-i}{H}\big]\\
    &=[\gamma_1(H)^{p^{i-1}},N,\prescript{}{n-i}{H}][\gamma_2(H)^{p^{i-2}},N,\prescript{}{n-i}{H}]\ldots [\gamma_i(H), N,\prescript{}{n-i}{H}].
\end{align*}

Hence it suffices to show $[\overline{\gamma_j(H)}^{p^{i-j}},\overline{N},\  _{n-i}\ H]=1$ for all $1\leq j\leq i$. By Theorem \ref{PR:T:1}, we have $[\overline{\gamma_j(H)}^{p^{i-j}},\overline{N}]\leqslant \prod\limits_{r=0}^{r=i-j}[\overline{N},\ _{p^{r}}\ \overline{\gamma_j(H)}]^{p^{i-j-r}}$. Since $[\overline{N},\ _{p^{r}}\ \overline{\gamma_j(H)}]\leqslant \overline{N}$, and $\overline{N}$ has exponent $p$, we have $[\overline{N},\ _{p^{r}}\ \overline{\gamma_j(H)}]^{p^{i-j-r}}=1$,  for all $r< i-j$. Thus\begin{align*}
    [\overline{\gamma_j(H)}^{p^{i-j}},\overline{N},\ _{n-i}\ H]&\leqslant [\overline{N},\ _{p^{i-j}}\ \overline{\gamma_j(H)},\ _{n-i}\ H]\\
    &\leqslant [Z_n(H),\ _{p^{i-j}}\ \overline{\gamma_j(H)},\ _{n-i}\ H]\\
    &\leqslant \overline{[Z_n(H),\ _{p^{i-j}}\ \gamma_j(H),\ _{n-i}\ H]}\\
    &\leqslant Z_{n-jp^{i-j}-n+i}(H) &&\mbox{(Lemma \ref{PR:R:1} \textit{(ii)})}
\end{align*} It is easy to see $jp^{i-j}\geq i$ for all $1\leq j\leq i$. Therefore  \[[\overline{\gamma_j(H)}^{p^{i-j}},\overline{N},\ _{n-i}\ H]\leqslant Z_0(H)=1,\] as required.
\par\textit{(ii)} By Lemma \ref{PR:R:2}, we have ${[N,\ _n\ H]\leqslant [Z_n(H),\ _n\  H]=1}$.
 \end{proof}
 \begin{corollary}
Let $H$ be a pro-$p$ group. If $N$ is a normal subgroup of $Z_n(H)$ with exponent $p$, then $[P_i(H), N,\ _{n-i}\ H]=1$ for $1\leq i\leq n$.
 \end{corollary}
 \begin{proof}
    This follows from lemma \ref{PR:L:1}, Lemma \ref{PR:R:2} and lemma \ref{MR:L:1}.
 \end{proof}
 
We need the next lemma crucially for the main theorem of this section.
 \begin{lemma}\label{MR:L:3}
		Let $H$ be a pro-$p$ group, and let $1 \to N \to H \to G \to 1$ be an extension of groups. Assume that $N$ is a closed normal subgroup of $H$, and the exponent of $N$ is $p$. If $N\leqslant Z_n(H)$ for a fixed positive integer $n$, then there exists a surjective homomorphism $\theta: G^{\otimes^{p}_{n+1}} \to \lambda_{n+1}(H)$.
 \end{lemma}
 \begin{proof}
		For $1\leq i\leq n+1$, set $N_i$ to be the image of the natural homomorphism $f_i^{n+1}:\hat{N}_{n+1}^{i}\to H^{\otimes^{p}_{n+1}}$ defined in Lemma \ref{{N}_i normality}. By Proposition \ref{nfold_qtensor_exact_sequence}, we have an exact sequence of groups 
        \[
        \begin{tikzcd}
		    &1 \arrow[r] &\prod \limits_{i=1}^{n+1} N_i \arrow[r] &H^{\otimes^{p}_{n+1}} \arrow[r] &G^{\otimes^{p}_{n+1}} \arrow[r] &1
		\end{tikzcd}
        \]
		 Applying Proposition \ref{The commutative cube}\textit{(iv)} to diagram \eqref{special cube for exact sequence for Lazard}, we get the following commutative diagram: 
		\begin{center}
			\begin{tikzcd} 
			
			H^{\otimes^{p}_{n+1}} \arrow[r]\arrow[d,"\mu ^{H}_{n+1}"] & G^{\otimes^{p}_{n+1}} \arrow[d,"\mu^{G}_{n+1}"]

			\\\ H \arrow[r]& G
			\end{tikzcd}
		\end{center}where $H^{\otimes^p_1}=H$, $\mu_1^H$ is identity as defined in \eqref{n fold qtensorsquare1}, and ${\mu_{n+1}^H: H^{\otimes_n^p}\otimes H\to H}$ is defined by
        \begin{align*}
            \mu_{n+1}^H(x \otimes h)=[\mu_{n}^H(x),h], && \mu_{n+1}^H\Big(\Big\{\big(x,\mu_{n}^H(x)\big)\Big\}\Big)={\mu_{n}^H(x)}^p
        \end{align*}
        for all $x \in H^{\otimes^p_{n}},h \in H$, as in \eqref{q n fold tensorsquare}. The map $\mu_{n+1}^G$ is defined similarly. Because $ \im{\mu_2^H}=\lambda_2(H)$ and $\im{\mu^H_{n+1}}= \im{\mu_n^H}^p[\im{\mu_n^H},H]$, we have that ${\im{\mu^H_{n+1}}=\lambda_{n+1}(H)}$. Hence to show the existence of a surjective homomorphism ${\theta: G^{\otimes^{p}_{n+1}} \to \lambda_{n+1}(H)}$, it is enough to show that  $\mu_{n+1}^H(N_i)=1$ for all $1\leq i\leq n+1$. By Lemma \ref{{N}_i normality}, we have the following commutative diagram  
  \begin{equation}
      	\begin{tikzcd} 	
			\hat{N}^i_{n+1} \arrow[r]\arrow[d,"\mu ^{i}_{n+1}"] & H^{\otimes^{p}_{n+1}} \arrow[d,"\mu^H_{n+1}"] 
			\\\ H \arrow[r,"id"]& H
			\end{tikzcd} \label{N_i commutativity}
  \end{equation}
  It is easy to see $\mu ^{i}_{n+1}(\hat{N}^1_{n+1})= [N,\ _n\ H]$, and $\mu ^{i}_{n+1}(\hat{N}^i_{n+1})= [\lambda_{i-1}(H), N,\ _{n-i}\ H]$ for all $i$ with $1< i\leq n+1 $. Using commutativity of \eqref{N_i commutativity} and Lemma \ref{MR:L:1}, we obtain $\mu_{n+1}^H(N_i)=\mu ^{i}_{n+1}(\hat{N}^i_{n+1})=1$, and hence the proof.

	\end{proof} 

\begin{corollary}\label{MR:T:1}
Let $H$ be a  pro-$p$ group and $N$ be an open normal subgroup of $H$. If $N\leqslant Z_n(H)$ for some positive integer $n$ and exponent of $N$ is $p$, then $\lambda_{n+1}(H)$ and $P_{n+1}(H)$ are finite $p$-groups.
\end{corollary}
\begin{proof}
Since $N$ is an open normal subgroup of $H$, we have that  $G=H/N$ is a finite $p$-group. By \cite[Corollary 12]{MR1088877}, the main results of \cite{E87}, and using induction, we obtain that $G^{\otimes^{p}_{n+1}}$ is a finite $p$-group. Using Lemma \ref{MR:L:3}, we have a surjective homomorphism $G^{\otimes^{p}_{n+1}}\to \lambda_{n+1}(H)$, and thus $\lambda_{n+1}(H)$ is a finite $p$-group. Since $H$ is Hausdorff, we obtain that $\lambda_{n+1}(H)= \overline{\lambda_{n+1}(H)}$. Now the proof follows from Lemma \ref{PR:L:1}.


\end{proof}

\section{Powerfulness and finiteness of the terms of the Frattini series for a pro-$p$ group} \label{frattini series}

    For a pro-$p$ group $H$, the Frattini series is defined as $\Phi_1(H):=G$ and  $\Phi_n(H)=\overline{\Phi_{n-1}(H)^p[\Phi_{n-1}(H),\Phi_{n-1}(H)]}$ for all $n>1$. In this section, we need another series related to the Frattini series. We define 
    \begin{align*}
        \Psi_1(H)=H \quad \text{and} \quad \Psi_n(H)=\Psi_{n-1}(H)^p[\Psi_{n-1}(H),\Psi_{n-1}(H)],
    \end{align*}
    for all $n>1$. By Lemma \ref{PR:R:2} and induction, $\overline{\Psi_n(H)}=\Phi_n(H)$ for all $n\geq 1$.

\begin{definition}\label{Frattini-normal subgroup}
    Let $H$ be a group. For each $n \geq 1$, we define a normal subgroup $\mathcal{F}_n(H)$ of $H$ by $$\mathcal{F}_n(H)=\{h \in H \mid [ \ldots [[h,x_1],x_2],\ldots,x_n]=1 \text{ for all }x_i \in \Psi_i(H) \}$$
\end{definition}

By setting $M=G=H_{\ptens_n}$ and $\mu=id$ in \eqref{Gqtensorcrossedsquare} we obtain the following crossed square for all $n \geq 1$
  \begin{equation} \label{n fold qtensorsquare_n}
\begin{tikzcd} 
		H_{\ptens_{n+1}}=H_{\ptens_n}\ptens H_{\ptens_n} \arrow[r,"\beta_{n+1}"] \arrow[d,"\alpha_{n+1}"] & H_{\ptens_n} \arrow[d,"id"] \\
		H_{\ptens_n} \arrow[r,"id"] & H_{\ptens_n}
  \end{tikzcd}
\end{equation}
By composing $\alpha_i$ for all $i$ with $2 \leq i \leq n+1$, we obtain the homomorphism ${\kappa_{n+1}^H : H_{\ptens_{n+1}} \to H}$ for all $n \geq 1$. It is easy to see that the image of $\kappa_{n+1}^H$ is $\Psi_{n+1}(H)$. Let $N$ be an $H$-invariant subgroup of $H$. We denote $\mathcal{N}_1^2$ to be the image of $N \ptens H$ in $H \ptens H$, and note that it is $H$-invariant. $\mathcal{N}_2^2$ is defined similarly. We inductively define $\mathcal{N}_i^{2^{n+1}}$ as the image of $\mathcal{N}_i^{2^n} \ptens H_{\ptens_n}$ in $H_{\ptens_{n+1}}$ if $i \leq 2^n$ and the image of $H_{\ptens_n} \ptens \mathcal{N}_{i-2^n}^{2^n} $ in $H_{\ptens_{n+1}}$ if $i > 2^n$.

  \begin{lemma}\label{Frattini exact sequence}
		Let $1 \to N \to H \to G \to 1$ be an extension of groups such that $N$ is a closed normal subgroup of $\overline{\mathcal{F}_n(H)}$ with exponent $p$ for a fixed positive integer $n$. Then there exists a surjective homomorphism ${\theta: G_{\otimes^{p}_{n+1}} \to \Psi_{n+1}(H)}$.
 \end{lemma}
 \begin{proof}
	Similar to Proposition \ref{nfold_qtensor_exact_sequence}, we have an exact sequence of groups
\[
     \begin{tikzcd}
         1 \arrow[r] & \prod \limits_{i=1}^{2^{n+1}} \mathcal{N}_i^{2^{n+1}} \arrow[r] & H_{\ptens_{n+1}} \arrow[r,"d"] & G_{\ptens_{n+1}} \arrow[r] & 1 
     \end{tikzcd}
\]
 
		Moreover by Proposition \ref{The commutative cube}, we have the following commutative diagram: 
\[			\begin{tikzcd} 
			H_{\ptens_{n+1}} \arrow[r]\arrow[d,"\kappa ^{H}_{n+1}"] & G_{\ptens_{n+1}} \arrow[d,"\kappa ^{G}_{n+1}"] 
			\\\ H \arrow[r]& G
			\end{tikzcd}
   \]
First recall that $\kappa_{n+1}^H(H_{\ptens_{n+1}})=\Psi_{n+1}(H)$. Hence to show the existence of a surjective homomorphism $\theta: G_{\otimes^{p}_{n+1}} \to \Psi_{n+1}(H)$, it suffices to show that $\kappa_{n+1}^H(\mathcal{N}_i^{2^{n+1}})=1$ for all $1\leq i\leq 2^{n+1}$. Using the fact that  ${\kappa_{n+1}=\kappa_n\circ \alpha_{n+1}}$, Lemma \ref{PR:R:1}\textit{(iii)} and induction on $n$, it can be easily shown that 
\begin{align*}
    \kappa_{n+1}^H(\mathcal{N}_i^{2^{n+1}})= [\ldots[[N,\Psi_1(H)],\Psi_2(H)],\ldots,\Psi_n(H)].
\end{align*} 
Since $N \leq \mathcal{F}_n(H)$, we have that   $\kappa_{n+1}^H(\mathcal{N}_i^{2^{n+1}})=1$, and hence the proof.
	\end{proof} 

 \begin{theorem} \label{Frattini is finite powerful}
Let $H$ be a  pro-$p$ group and $N$ be an open normal subgroup of $H$. Assume $N\leqslant \overline{\mathcal{F}_n(G)}$ for some positive integer $n$ and exponent of $N$ is $p$. The following statements hold:
\begin{enumerate}
    \item [(i)] $\Psi_{n+1}(H)$ and $\Phi_{n+1}(H)$ are finite $p$-groups. 
 \item [(ii)] Let $p$ be an odd prime. If  $H/N$ is powerful, then both $\Psi_{n+1}(H)$ and $\Phi_{n+1}(H)$ are finite powerful $p$-groups.
\end{enumerate}
\end{theorem}
\begin{proof}
\textit{(i)} Since $N$ is an open normal subgroup of $H$, we have that  $G=H/N$ is a finite $p$-group. By \cite[Corollary 12]{MR1088877}, the main results of \cite{E87}, and using induction, we obtain that $G_{\otimes^{p}_{n+1}}$ is a finite $p$-group. Moreover by Lemma \ref{Frattini exact sequence}, we have a surjective homomorphism $G_{\otimes^{p}_{n+1}}\to \Psi_{n+1}(H)$, and thus $\Psi_{n+1}(H)$ is a finite $p$-group. Thus ${\Phi_{n+1}(H)=\overline{\Psi_{n+1}(H)}=\Psi_{n+1}(H)}$ is a finite $p$-group.
\par\textit{(ii)} Since $G=H/N$ is finite powerful $p$-group, we obtain that $G_{\otimes^{p}_{n+1}}$ is a finite powerful $p$-group by Corollary \ref{n-iterated qtensor product is powerful}. Now the proof follows \textit{mutatis mutandis} the proof of $\textit{(i)}$.
\end{proof}

\section{Applications to automorphism groups of powerful $p$-groups} \label{section: General version of commutator and Frattini are powerful}
Let $G$ be a finite $p$-group. To develop the applications to automorphism groups mentioned in the introduction, we consider the following subgroups of $\operatorname{Aut}(G)$:
\begin{align*}
    \operatorname{Aut}_\Phi(G) &= \bigl\{\, \alpha \in \operatorname{Aut}(G) \;\big|\; \alpha(g)g^{-1} \in \Phi(G) \;\forall\, g \in G \,\bigr\}\\
&= \ker\!\Bigl(\operatorname{Aut}(G) \longrightarrow \operatorname{Aut}(G/\Phi(G))\Bigr).
\end{align*}
We also consider the following two well known subgroups of $\operatorname{Aut}(G)$, namely the class preserving automorphisms $\operatorname{Aut}_c(G)$, and $\operatorname{IA}(G)$, defined by 
\begin{align*}
    \operatorname{IA}(G) &= \bigl\{\, \alpha \in \operatorname{Aut}(G) \;\big|\; \alpha(g)g^{-1} \in [G, G] \,\bigr\}. 
\end{align*}
Note that $\operatorname{Inn} G \leqslant \operatorname{Aut}_c(G) \leqslant \operatorname{IA}(G)\leqslant \operatorname{Aut}_{\Phi}(G)$. Since $\operatorname{Aut}_\Phi(G)$ is a $p$-group by \cite[Proposition 3.29]{Serr2022} (see also \cite[Theorem 3.18]{Hup2025}), it follows that each of the three subgroups of $\operatorname{Aut}(G)$ considered here is a $p$-group.

We also define the inner automorphism group of $G$ induced by elements of a subgroup $H$ of $G$ as
\begin{align*}
    \operatorname{Inn}_H(G) &= \bigl\{\, f_h \in \operatorname{Inn}(G) \;\big|\; h \in H\,\bigr\}.
\end{align*}
We first show that $\operatorname{Inn}(G)$ is powerfully embedded in $\operatorname{Aut}_\Phi(G)$.
\begin{lemma} \label{InnG powerfully embedded in Aut_Phi(G)}
Let $G$ be a finite $p$-group. If $G$ is powerful, then $\operatorname{Inn}(G)$ is powerfully embedded in $\operatorname{Aut}_\Phi(G)$. 
\end{lemma}

\begin{proof}
Let $g \in G$, and let $f_g$ be the inner automorphism defined by $g$. Note that $[f_g,\, \alpha]= f_{g\alpha(g^{-1})}$, for all $f_g \in \operatorname{Inn}(G)$ and $\alpha \in \operatorname{Aut}_\Phi(G)$.
Hence ${[f_g, \alpha] \in \operatorname{Inn}_{\Phi(G)}(G)}$. 
Since $\phi(G) = G^p$, it follows that
\[
[\operatorname{Inn} G,\, \operatorname{Aut}_\Phi(G)] \leqslant \operatorname{Inn}_{G^p}(G).
\]
We claim that $\operatorname{Inn}_{G^p}(G) 
\subseteq \operatorname{Inn}(G)^p$. Towards that end, let $h \in G^p$. Since $G$ is powerful, we have that $h=x^p$ for some $x\in G$. Thus ${f_h(g)=f_{x^p}(g)=(f_{x})^p(g)}$. Therefore
\[
[\operatorname{Inn} G,\, \operatorname{Aut}_\phi(G)] \leqslant \operatorname{Inn}(G)^p.
\]
\end{proof} 

To prove that $\operatorname{Inn}(G)$ is powerfully embedded in $\operatorname{IA}(G)$ and $\operatorname{Aut}_c(G)$, we utilize the following lemma. Its proof is straightforward and is therefore omitted. 
\begin{lemma}\label{transitivity of PE}
Let $G$ be a finite $p$-group. Suppose $H, K$ are subgroups of $G$ with $H \leqslant K \leqslant G$. If $H$ is powerfully embedded in  $G$, then $H$ is powerfully embedded in $K$.
\end{lemma}

\begin{lemma} \label{Inn G pe}
Let $G$ be a finite powerful $p$-group. Then $\operatorname{Inn} G$ is powerfully embedded in $\operatorname{IA}(G)$, and furthermore, $\operatorname{Inn} G $ is powerfully embedded in $\operatorname{Aut}_c(G)$.
\end{lemma}

\begin{proof}
The result follows from Lemmas \ref{InnG powerfully embedded in Aut_Phi(G)} and \ref{transitivity of PE}.
\end{proof}

We now show that the non-abelian tensor products of $G$ and these automorphism groups are powerful.
\begin{corollary} \label{Aut tensor powerful}
    Let $G$ be a finite $p$-group. If $G$ is powerful, then the following groups are powerful:
    \begin{enumerate}
        \item[(i)] $G \otimes \operatorname{Aut}_\Phi(G)$.
        \item[(ii)] $G \otimes \operatorname{Aut}_c(G)$.
        \item[(iii)] $G \otimes 
        \operatorname{IA} G$.
        \item[(iv)] $G \ptens \operatorname{Aut}_\Phi(G)$.
        \item[(v)] $G \ptens \operatorname{Aut}_c(G)$.
        \item[(vi)] $G \ptens 
        \operatorname{IA} G$.
    \end{enumerate}
    
\end{corollary}

\begin{proof}
    We prove $(i)$, and the rest follows similarly. Let $\mu: G \to \operatorname{Aut}_\Phi(G)$ be the crossed module which maps $g$ to the inner automorphism defined by $g$. Note that the image $\mu(G)$ is $\operatorname{Inn} G$. By Lemma \ref{Inn G pe}, $\operatorname{Inn} G$ is powerfully embedded in $\operatorname{Aut}_\Phi(G)$. Now, the result follows from Theorems \ref{T:Tensor of poweful crossed module} and \ref{qtensorpowerful}.
\end{proof}

We now show that if $G$ is a finite powerful $p$-group, then the generalized commutator and Frattini subgroups discussed in the introduction are also powerful.
\begin{corollary} \label{Corollary: Generalizations of Commutator and Frattini are powerful}
    Let $G$ be a finite $p$-group. If $G$ is powerful, then $L_A$ and $
    \Phi_A$ are powerful for $A=\operatorname{Aut}_\Phi(G), \operatorname{Aut}_c(G), and \operatorname{IA}(G)$.
\end{corollary}

\begin{proof}
    By Corollary \ref{Aut tensor powerful}, $G \otimes \operatorname{Aut}_\Phi(G)$ and $G \ptens \operatorname{Aut}_\Phi(G)$ are powerful. Recalling the homomorphism $\alpha$ defined in Lemma \ref{BL identity} and \eqref{E:alpha}, it follows that the images $\alpha( G \otimes \operatorname{Aut}_\Phi(G))$ and $\alpha( G \ptens \operatorname{Aut}_\Phi(G))$ are powerful. Since $L_{\operatorname{Aut}_\Phi(G)} = \alpha( G \otimes \operatorname{Aut}_\Phi(G))$ and $\Phi_{\operatorname{Aut}_\Phi(G)} = \alpha( G \ptens \operatorname{Aut}_\Phi(G))$ by definition, the result follows. The proof that the remaining subgroups are powerful follow similarly.
\end{proof}

\section*{Acknowledgements} 
Sathasivam K acknowledges the Ministry of Education,  Government of India, for the doctoral fellowship under the Prime Minister's Research Fellows (PMRF) scheme PMRF-ID 0801996. V. Z. Thomas acknowledges research support from ANRF, Government of India grant CRG/2023/004417.

\bibliographystyle{amsplain}
\bibliography{Lazardsbib}
\end{document}